\documentclass[12pt,twoside,leqno]{article}
\usepackage[T1]{fontenc}
\usepackage{amsfonts}
\usepackage{amsmath,amsthm}
\usepackage{amssymb,latexsym}
\usepackage{enumerate}

\theoremstyle{plain}
\newtheorem{theorem}{Theorem}[section]
\newtheorem{lemma}[theorem]{Lemma}
\newtheorem{proposition}[theorem]{Proposition}
\newtheorem{corollary}[theorem]{Corollary}

\theoremstyle{definition}

\newtheorem{definition}[theorem]{Definition}
\theoremstyle{remark}
\newtheorem{remark}[theorem]{Remark}
\newtheorem{question}[theorem]{Question}

\DeclareMathOperator{\fix}{{\rm fix}}

\DeclareMathOperator{\dom}{{\rm dom}}
\DeclareMathOperator{\ran}{{\rm ran}}
\DeclareMathOperator{\HS}{{\rm HS}}
\DeclareMathOperator{\Fn}{{\rm Fn}}
\DeclareMathOperator{\op}{{\rm op}}

\DeclareMathOperator{\TC}{{\rm TC}}

\DeclareMathOperator{\sym}{{\rm sym}}

\begin{document}
\title{Several amazing discoveries about compact metrizable spaces in $\mathbf{ZF}$}
\author{Kyriakos Keremedis, Eleftherios Tachtsis and Eliza Wajch\\
Department of Mathematics, University of the Aegean\\
Karlovassi, Samos 83200, Greece\\
kker@aegean.gr\\
Department of Statistics and Actuarial-Financial Mathematics,\\
 University of the Aegean, Karlovassi 83200, Samos, Greece\\
 ltah@aegean.gr\\
Institute of Mathematics\\
Faculty of Exact and Natural Sciences \\
Siedlce University of Natural Sciences and Humanities\\
ul. 3 Maja 54, 08-110 Siedlce, Poland\\
eliza.wajch@wp.pl}
\maketitle
\begin{abstract}
In the absence of the axiom of choice, the set-theoretic status of many natural statements about metrizable compact spaces is investigated. Some of the statements are provable in $\mathbf{ZF}$, some are shown to be independent of $\mathbf{ZF}$. For independence results, distinct models of $\mathbf{ZF}$ and permutation models of $\mathbf{ZFA}$ with transfer theorems of Pincus are applied. New symmetric models are constructed in each of which the power set of $\mathbb{R}$ is well-orderable, the Continuum Hypothesis is satisfied but a denumerable family of non-empty finite sets can fail to have a choice function, and a compact metrizable space need not be embeddable into the Tychonoff cube $[0, 1]^{\mathbb{R}}$. \medskip

\end{abstract}

\section{Preliminaries}
\label{s1}
\subsection{The set-theoretic framework}
\label{s1.1}

In this paper, the intended context for reasoning and statements of theorems
is the Zermelo-Fraenkel set theory $\mathbf{ZF}$ without the axiom of choice 
$\mathbf{AC}$. The system $\mathbf{ZF+AC}$ is denoted by $\mathbf{ZFC}$. We recommend  \cite{ku1} and \cite{Ku} as a
good introduction to $\mathbf{ZF}$. To stress the fact that a result is proved in $\mathbf{ZF}$ or $\mathbf{ZF+A}$ (where $\mathbf{A}$ is a statement independent of $\mathbf{ZF}$), we shall write
at the beginning of the statements of the theorems and propositions ($%
\mathbf{ZF}$) or ($\mathbf{ZF+A}$), respectively. Apart from models of $\mathbf{ZF}$, we refer to some models
of $\mathbf{ZFA}$ (or $\text{ZF}^0$ in \cite{hr}), that is, we refer also to $\mathbf{ZF}$ with an infinite set of atoms (see \cite{j}, \cite{j1} and \cite{hr}). Our theorems proved here in $\mathbf{ZF}$ are also provable in $\mathbf{ZFA}$; however, we also mention some theorems of $\mathbf{ZF}$ that are not theorems of $\mathbf{ZFA}$. 

We denote by $\omega$ the set of all non-negative integers (i.e., finite ordinal numbers of von Neumann). As usual, if $n\in\omega$, then $n+1=n\cup\{n\}$. Members of the set $\mathbb{N}=\omega\setminus\{0\}$ are called natural numbers. The power set of a set $X$ is denoted by $\mathcal{P}(X)$. A set $X$ is called \emph{countable} if $X$ is equipotent to a subset of $\omega$. A set $X$ is called \emph{uncountable} if $X$ is not countable.  A set $X$ is \emph{finite} if $X$ is equipotent to an element of $\omega$. An \emph{infinite} set is a set which is not finite. An infinite countable set is called \emph{denumerable}. A \emph{cardinal number of von Neumann} is an initial ordinal number of von Neumann. If $X$ is a set and $\kappa$ is a non-zero cardinal number of von Neumann, then $[X]^{\kappa}$ is the family of all subsets of $X$ equipotent to $\kappa$, $[X]^{\leq\kappa}$ is the collection of all subsets of $X$ equipotent to subsets of $\kappa$, and $[X]^{<\kappa}$ is the family of all subsets of $X$ equipotent to a cardinal number of von Neumann which belongs to $\kappa$. For a set $X$, we denote by $|X|$ the cardinal number of $X$ in the sense of Definition 11.2 of \cite{j}. We recall that, in $\mathbf{ZF}$, for every set $X$, the cardinal number $|X|$ exists; however, X is equipotent to a cardinal number of von Neumann if and only if $X$ is well-orderable.  For sets $X$ and $Y$, the inequality $|X|\leq |Y|$ means that $X$ is equipotent to a subset of $Y$.

The set of all real numbers is denoted by $\mathbb{R}$ and, if it is not stated otherwise, $\mathbb{R}$ and every subspace of $\mathbb{R}$ are considered with the usual topology and with the metric induced by the standard absolute value on $\mathbb{R}$.

\subsection{Notation and basic definitions}
\label{s1.2}

In this subsection, we establish notation and recall several basic definitions.  

Let $\mathbf{X}=\langle X, d\rangle$ be a metric space.  The $d$-\textit{ball with centre $x\in X$ and radius} $r\in(0, +\infty)$ is the set 
$$B_{d}(x, r)=\{ y\in X: d(x, y)<r\}.$$
 The collection 
$$\tau(d)=\{ V\subseteq X: (\forall x\in V)(\exists n\in\omega) B_{d}(x, \frac{1}{2^n})\subseteq V\}$$
is the \textit{topology in $X$ induced by $d$}.  For a set $A\subseteq X$, let $\delta_d(A)=0$ if $A=\emptyset$, and let $\delta_d(A)=\sup\{d(x,y): x,y\in A\}$ if $A\neq \emptyset$. Then $\delta_d(A)$ is the \emph{diameter} of $A$ in $\mathbf{X}$.

\begin{definition}
\label{s1d1}
Let $\mathbf{X}=\langle X, d\rangle$ be a metric space.
\begin{enumerate}
\item[(i)] Given a real number $\varepsilon>0$, a subset $D$ of $X$ is called $\varepsilon$-\emph{dense} or an $\varepsilon$-\emph{net} in $\mathbf{X}$ if $X=\bigcup\limits_{x\in D}B_d(x, \varepsilon)$. 
\item[(ii)]  $\mathbf{X}$ is called  \emph{totally bounded} if, for every real number $\varepsilon>0$, there exists a finite $\varepsilon$-net in $\mathbf{X}$.
\item[(iii)] $\mathbf{X}$ is called  \emph{strongly totally bounded} if it admits a sequence $(D_n)_{n\in\mathbb{N}}$ such that, for every $n\in\mathbb{N}$, $D_n$ is a finite $\frac{1}{n}$-net in $\mathbf{X}$.
\item[(iv)] (Cf. \cite{kyri}.) $d$ is called \emph{strongly  totally bounded} if $\mathbf{X}$ is strongly totally bounded.
\end{enumerate}
\end{definition}

\begin{remark}
\label{s1r2}
 Every strongly totally bounded metric space is evidently totally bounded. However, it was shown in \cite[Proposition 8]{kyri} that the sentence ``Every totally bounded metric space is strongly totally bounded'' is not a theorem of $\mathbf{ZF}$.
\end{remark}

\begin{definition}
\label{s1d3}
Let  $\mathbf{X}=\langle X, \tau\rangle$  be a topological space and let  $Y\subseteq X$. Suppose that $\mathcal{B}$ is a base of $\mathbf{X}$.
\begin{enumerate}
\item[(i)] The closure of $Y$ in $\mathbf{X}$ is denoted by $\text{cl}_{\tau}(Y)$ or $\text{cl}_{\mathbf{X}}(Y)$.
\item[(ii)] $\tau|_Y=\{U\cap Y: U\in\tau\}$. $\mathbf{Y}=\langle Y, \tau|_ Y\rangle$ is the subspace of $\mathbf{X}$ with the underlying set $Y$.
\item[(iii)]  $\mathcal{B}_Y=\{U\cap Y: U\in\mathcal{B}\}$.
\end{enumerate}
\end{definition}

Clearly, in Definition \ref{s1d3} (iii), $\mathcal{B}_Y$ is a base of $\mathbf{Y}$. In Section \ref{s5}, it is shown that $\mathcal{B}_Y$ need not be equipotent to a subset of $\mathcal{B}$.

 In the sequel, boldface letters will denote metric or topological spaces (called spaces in abbreviation) and lightface letters will denote their underlying sets.

\begin{definition}
\label{s1d4} A collection $\mathcal{U}$ of subsets of a space $\mathbf{X}$ is called:
\begin{enumerate} 
\item[(i)] \emph{locally finite} if every point of $X$ has a neighbourhood meeting only finitely many members of $\mathcal{U}$;
\item[(ii)] \emph{point-finite} if every point of $X$ belongs to at most finitely many members of $\mathcal{U}$;
\item[(iii)] $\sigma$-\emph{locally finite} (respectively, $\sigma$-\emph{point-finite}) if $\mathcal{U}$ is a countable union of locally finite (respectively, point-finite)  subfamilies. 
\end{enumerate}
\end{definition}

\begin{definition}
\label{s1d5} A space $\mathbf{X}$ is called:
\begin{enumerate}
 \item[(i)]  \emph{first-countable} if every point of $X$ has a countable base of neighbourhoods;
 \item[(ii)] \emph{second-countable} if $\mathbf{X}$ has a countable base.
 \end{enumerate}
\end{definition}

Given a collection  $\{X_j: j\in J\}$ of sets, for every $i\in J$, we denote by $\pi_i$ the projection $\pi_i:\prod\limits_{j\in J}X_j\to X_i$ defined by $\pi_i(x)=x(i)$ for each $x\in\prod\limits_{j\in J}X_j$. If $\tau_j$ is a topology in $X_j$, then $\mathbf{X}=\prod\limits_{j\in J}\mathbf{X}_j$ denotes the Tychonoff product of the topological spaces $\mathbf{X}_j=\langle X_j, \tau_j\rangle$ with $j\in J$. If $\mathbf{X}_j=\mathbf{X}$ for every $j\in J$, then $\mathbf{X}^{J}=\prod\limits_{j\in J}\mathbf{X}_j$. As in \cite{En}, for an infinite set $J$ and the unit interval $[0,1]$ of $\mathbb{R}$, the cube $[0,1]^J$ is called the \emph{Tychonoff cube}. If $J$ is denumerable, then the Tychonoff cube $[0,1]^J$ is called the \emph{Hilbert cube}. In \cite{hh}, all Tychonoff cubes are called Hilbert cubes. In \cite{w}, Tychonoff cubes are called cubes. 

We recall that if $\prod\limits_{j\in J}X_j\neq\emptyset$, then it is said that the family $\{X_j: j\in J\}$ has a choice function, and every element of $\prod\limits_{j\in J}X_j$ is called a \emph{choice function} of the family $\{X_j: j\in J\}$. A \emph{multiple choice function} of $\{X_j: j\in J\}$ is every function $f\in\prod\limits_{j\in J}\mathcal{P}(X_j)$ such that, for every $j\in J$, $f(j)$ is a non-empty finite subset of $X_j$. A set $f$ is called \emph{partial} (\emph{multiple}) \emph{choice function} of $\{X_j: j\in J\}$ if there exists an infinite subset $I$ of $J$ such that $f$ is a (multiple) choice function of $\{X_j: j\in I\}$. Given a non-indexed family $\mathcal{A}$, we treat $\mathcal{A}$ as an indexed family $\mathcal{A}=\{x: x\in\mathcal{A}\}$ to speak about a  (partial) choice function and a (partial) multiple choice function of $\mathcal{A}$.

Let  $\{X_j: j\in J\}$ be a disjoint family of sets, that is, $X_i\cap X_j=\emptyset$ for each pair $i,j$ of distinct elements of $J$. If $\tau_j$ is a topology in $X_j$ for every $j\in J$, then $\bigoplus\limits_{j\in J}\mathbf{X}_j$ denotes the direct sum of the spaces $\mathbf{X}_j=\langle X_j, \tau_j\rangle$ with $j\in J$.  

\begin{definition}
\label{s1d6}
(Cf. \cite{br}, \cite{lo} and \cite{kerta}.) 
\begin{enumerate}
\item[(i)] A space $\mathbf{X}$ is said to be \emph{Loeb} (respectively, \emph{weakly Loeb}) if the family of all non-empty closed subsets of $\mathbf{X}$ has a choice function (respectively, a multiple choice function).
\item[(ii)] If $\mathbf{X}$ is a (weakly) Loeb space, then every (multiple) choice function of the family of all non-empty closed subsets of $\mathbf{X}$ is called a (\emph{weak}) \emph{Loeb function} of $\mathbf{X}$.
\end{enumerate}
\end{definition}

Other topological notions used in this article but not defined here are standard. They can be found, for instance, in \cite{En} and \cite{w}. 

\begin{definition}
\label{s1d7} A set $X$ is called:
\begin{enumerate}
\item[(i)] a \emph{cuf set} if $X$ is expressible as a countable union of finite sets (cf. \cite{hdkr}, \cite{hdhkr}, \cite{HowTach} and \cite[Form 419]{hr1});
\item[(ii)] \emph{Dedekind-finite} if $X$ is not equipotent to a proper subset of itself (cf. \cite[Note 94]{hr}, \cite[Definition 4.1]{hh} and \cite[Definition 2.6]{j}); \emph{Dedekind-infinite} if $X$ is not Dedekind-finite (equivalently, if there exists an injection $f:\omega\to X$) (cf. \cite[Note 94]{hr} and  \cite[Definition 2.13]{hh});
\item[(iii)]  \emph{amorphous} if $X$ is infinite and there does not exist a partition of $X$ into two infinite sets (cf. \cite[Note 57]{hr},\cite[p. 52]{j} and  \cite[E. 11 in Section 4.1]{hh}).
\end{enumerate}
\end{definition}

\begin{definition}
\label{s1d8}
(Cf. \cite{kw1}.) A topological space $\langle X, \tau\rangle$ is called a \emph{cuf space} if $X$ is a cuf set. 
\end{definition}

\subsection{The list of weaker forms of $\mathbf{AC}$}
\label{s1.3}

In this subsection, for readers' convenience, we define and denote most of the weaker forms of $\mathbf{AC}$ used directly in this paper. If a form is not defined in the forthcoming sections, its definition can be found in this subsection. For the known forms given in \cite{hr}, \cite{hr1} or \cite{hh}, we quote in their statements the form number under which they are recorded in \cite{hr} (or in \cite{hr1} if they do not appear in \cite{hr}) and, if possible, we refer to their definitions in \cite{hh}. 

\begin{definition}
\label{s1d9}
\begin{enumerate}
\item $\mathbf{AC}_{fin}$ (\cite[Form 62]{hr}): Every non-empty family of non-empty finite sets has a choice function.
\item $\mathbf{AC}_{WO}$ ( \cite[Form 60]{hr}): Every non-empty family of non-empty well-orderable sets has a choice function. 
\item $\mathbf{CAC}$ (\cite[Form 8]{hr}, \cite[Definition 2.5]{hh}): Every denumerable family of non-empty sets has a choice function.

\item  $\mathbf{CAC}(\mathbb{R})$ (\cite[Form 94]{hr}, \cite[Definition 2.9(1)]{hh}): Every denumerable family of non-empty subsets of $\mathbb{R}$ has a choice function.

\item $\mathbf{CAC}_{\bigtriangleup \omega }(\mathbb{R})$ (Cf. \cite{keremTac}): For every family $%
\mathcal{A}=\{A_n: n\in\omega\}$ such that, for every $n\in\omega$ and all $x,y\in A_n$, $\emptyset\neq A_n\subseteq\mathcal{P}(\omega)\setminus\{\emptyset\}$ and $x\bigtriangleup y\in[\omega]^{<\omega} $ ($\bigtriangleup $
denotes the operation of symmetric difference between sets), there exists a choice function of $\mathcal{A}$.

\item $\mathbf{IDI}$ (\cite[Form 9]{hr}, \cite[Definition 2.13(ii)]{hh}): Every Dedekind-finite set is finite.

\item $\mathbf{IDI}(\mathbb{R})$ (\cite[Form 13]{hr}, \cite[Definition 2.13(2)]{hh}): Every Dedekind-finite subset of $\mathbb{R}$ is finite.

\item $\mathbf{WoAm}$ (\cite[Form 133]{hr}): Every set  is either well-orderable or has an amorphous subset.

\item $\mathbf{Part}(\mathbb{R})$: Every partition of $\mathbb{R}$ is of size $\leq |\mathbb{R}|$.

\item $\mathbf{WO}(\mathbb{R})$ (\cite[Form 79]{hr}): $\mathbb{R}$ is well-orderable. 

\item $\mathbf{WO}(\mathcal{P}(\mathbb{R}))$ (\cite[Form 130]{hr}): $\mathcal{P}(\mathbb{R})$ is well-orderable. 

\item $\mathbf{CAC}_{fin}$ (\cite[Form 10]{hr}, \cite[Definition 2.9(3)]{hh}): Every denumerable family of non-empty finite sets has a choice function.

\item For a fixed $n\in\omega\setminus \{0, 1\}$, $\mathbf{CAC}_n$ (\cite[Form 288(n)]{hr}): Every denumerable family of $n$-element sets has a choice function. 

\item $\mathbf{CAC}_{WO}$: Every denumerable family of non-empty well-orderable sets has a choice function.

\item $\mathbf{CMC}$ (\cite[Form 126]{hr}, \cite[Definition 2.10]{hh}): Every denumerable family of non-empty sets has a multiple choice function.

\item $\mathbf{CMC}_{\omega }$ (\cite[Form 350]{hr}): Every denumerable family of denumerable sets has a multiple choice function.

\item $\mathbf{CUC}$ (\cite[Form 31]{hr}, \cite[Definition 3.2(1)]{hh}): Every countable union of
countable sets is countable.

\item $\mathbf{CUC}_{fin}$ (\cite[Form \text{[10 A]}]{hr}, \cite[Definition 3.2(3)]{hh}): Every countable union of finite sets is countable.

\item $\mathbf{UT}(\aleph_0, cuf, cuf)$ (\cite[Form 419]{hr1}): Every countable union of cuf sets is a cuf set. (Cf. also \cite{hdhkr}.)

\item $\mathbf{UT}(\aleph_0, \aleph_0, cuf)$ (\cite[Form 420]{hr1}): Every countable union of countable sets is a cuf set. (Cf. also \cite{hdhkr}.)

\item $\mathbf{vDCP}(\aleph_0)$ ( \cite[Form 119]{hr}, \cite[p, 79]{hh}, \cite{vD}): Every denumerable family $\{\langle A_n, \leq_n\rangle: n\in\omega\}$ of linearly ordered sets, each of which is order-isomorphic to the set $\langle \mathbb{Z}, \leq\rangle$ of integers with the standard linear order $\leq$, has a choice function.

\item $\mathbf{BPI}$ (\cite[Form 14]{hr}, \cite[Definition 2.15(1)]{hh}):  Every Boolean algebra has a prime ideal. 

\item $\mathbf{DC}$ (\cite[Form 43]{hr}, \cite[Definition 2.11(1)]{hh}): For every non-empty set $X$ and every binary relation $\rho$ on $X$ if, for each $x\in X$ there exists $y\in X$ such that $x\rho y$, then  there exists a sequence $(x_n)_{n\in\mathbb{N}}$ of points of $X$ such that $x_n\rho x_{n+1}$ for each $n\in\mathbb{N}$.

\end{enumerate}
\end{definition}

\begin{remark}
\label{s1r10} 
The following are well-known facts in $\mathbf{ZF}$:
\begin{enumerate}
\item[(i)] $\mathbf{CAC}_{fin}$ and $\mathbf{CUC}_{fin}$ are both equivalent to the sentence: Every infinite well-ordered family of
non-empty finite sets has a partial choice function (see  \cite[Form \text{[10 O]}]{hr} and  \cite[Diagram 3.4, p. 23]{hh}). Moreover, 
 $\mathbf{CAC}_{fin}$ is equivalent to Form [10 E] of \cite{hr}, that is, to the sentence: Every denumerable family of non-empty finite sets has a partial choice function. It is known that $\mathbf{IDI}$ implies $\mathbf{CAC}_{fin}$ and this implication is not reversible in $\mathbf{ZF}$ (cf. \cite[pp. 324--324]{hh}).
 
\item[(ii)]  $\mathbf{CAC}$ is equivalent to the sentence: Every denumerable family of non-empty sets has a partial choice function (see  \cite[Form \text{[8 A]}]{hr}).

\item[(iii)] $\mathbf{BPI}$ is equivalent to the statement that all products of compact Hausdorff spaces are compact (see \cite[Form \text{[14 J]}]{hr} and \cite[Theorem 4.37]{hh}). 

\item[(iv)] $\mathbf{CMC}_{\omega}$ is equivalent to the following sentence: Every denumerable family of denumerable sets has a multiple choice function. 
\end{enumerate}
\end{remark}

\begin{remark}
\label{s1r11}
\begin{enumerate}
\item[(a)] It was proved in \cite{HowTach} that the following implications are true in $\mathbf{ZF}$ and none of the implications is reversible in $\mathbf{ZF}$:
$$\mathbf{CMC}\rightarrow\mathbf{UT}(\aleph_0, cuf,cuf)\rightarrow\mathbf{CMC}_{\omega}\rightarrow\mathbf{vDCP}(\aleph_0).$$

\item[(b)] Clearly, $\mathbf{UT}(\aleph_0, cuf, cuf)$ implies $\mathbf{UT}(\aleph_0, \aleph_0, cuf)$. In \cite[proof to Theorem 3.3]{hdhkr} a model of $\mathbf{ZFA}$ was shown in which  $\mathbf{UT}(\aleph_0, \aleph_0, cuf)$ is true and $\mathbf{UT}(\aleph_0, cuf,cuf)$ is false.

\item[(c)] It was proved in \cite{kw1} that the following equivalences hold in $\mathbf{ZF}$:
\begin{enumerate}
\item[(i)] $\mathbf{UT}(\aleph_0, cuf, cuf)$ is equivalent to the sentence: Every countable product of one-point Hausdorff compactifications of infinite discrete cuf spaces is metrizable (equivalently, first-countable).
\item[(ii)] $\mathbf{UT}(\aleph_0, \aleph_0, cuf)$ is equivalent to the sentence: Every countable product of one-point Hausdorff compactifications of denumerable discrete spaces is metrizable (equivalently, first-countable).
\end{enumerate}
\end{enumerate}
\end{remark}

Let us pass to definitions of forms concerning metric and metrizable spaces.

\begin{definition}
\label{s1d12}
\begin{enumerate}

\item $\mathbf{CAC}(\mathbb{R},C)$: For every disjoint family $%
\mathcal{A}=\{A_{n}:n\in \mathbb{N}\}$ of non-empty subsets of $\mathbb{R}$,
if there exists a family $\{d_{n}:n\in \mathbb{N}\}$ of metrics such that, for every $n\in \mathbb{%
N},\langle A_{n},d_{n}\rangle$ is a compact metric space, then $\mathcal{A}$ has a choice function.

\item $\mathbf{CAC}(C, M)$: If $\{\langle X_n, d_n\rangle: n\in\omega\}$ is a family of non-empty compact metric spaces, then the family $\{X_n: n\in\omega\}$ has a choice function.

\item $\mathbf{M}(TB, WO)$: For every totally bounded metric space $\langle X, d\rangle$, the set $X$ is well-orderable.

\item $\mathbf{M}(TB, S)$: Every totally bounded metric space is separable.

\item $\mathbf{ICMDI}$: Every infinite compact metrizable space is Dedekind-infinite.

\item $\mathbf{MP}$ (\cite[Form 383]{hr}): Every metrizable space is paracompact.

\item $\mathbf{M}(\sigma-p.f.)$ (\cite[Form 233]{hr}): Every metrizable space has a $\sigma$-point-finite base.

\item $\mathbf{M}(\sigma-l.f.)$ (\cite[Form \text{[232 B]}]{hr}): Every metrizable space has a $\sigma$-locally finite base.
\end{enumerate}
\end{definition} 

\begin{definition}
\label{s1d13}
The following forms will be called \emph{forms of type} $\mathbf{M}(C,\square)$.
\begin{enumerate}
\item $\mathbf{M}(C,S)$: Every compact metrizable space is separable.

\item $\mathbf{M}(C,2)$: Every compact metrizable space is second-countable.

\item $\mathbf{M}(C, STB)$: Every compact metric space is strongly totally bounded.

\item $\mathbf{M}(C,L)$: Every compact metrizable space is Loeb.

\item $\mathbf{M}(C, WO)$: Every compact metrizable space is well-orderable. 

\item $\mathbf{M}(C,\hookrightarrow \lbrack 0,1]^{\mathbb{N}})$: Every
compact metrizable space is embeddable in the Hilbert cube $[0,1]^{\mathbb{N}}$.

\item $\mathbf{M}(C,\hookrightarrow \lbrack 0,1]^{\mathbb{R}})$: Every
compact metrizable space is embeddable in the Tychonoff cube $[0,1]^{\mathbb{R}}$.

\item $\mathbf{M}(C,\leq |\mathbb{R}|)$: Every compact metrizable space is of size $\leq |\mathbb{R}|$.

\item $\mathbf{M}(C, W(\mathbb{R}))$: For every infinite compact metrizable space $\langle X, \tau\rangle$, $\tau$ and $\mathbb{R}$ are equipotent.

\item $\mathbf{M}(C, B(\mathbb{R}))$: Every compact metrizable space has a base of size $\leq|\mathbb{R}|$. 

\item $\mathbf{M}(C, |\mathcal{B}_Y|\leq |\mathcal{B}|)$: For every compact metrizable space $\mathbf{X}$, every base $\mathcal{B}$ of $\mathbf{X}$ and every compact subspace $\mathbf{Y}$ of $\mathbf{X}$,  $|\mathcal{B}_Y|\leq|\mathcal{B}|$.

\item  $\mathbf{M}([0,1], |\mathcal{B}_Y|\leq |\mathcal{B}|)$: For every base $\mathcal{B}$  of the interval $[0,1]$ with the usual topology and every compact subspace $\mathbf{Y}$ of $[0,1]$, $|\mathcal{B}_Y|\leq|\mathcal{B}|$.

\item $\mathbf{M}(C, \sigma-l.f)$: Every compact metrizable space has a $\sigma$-locally finite base.

\item $\mathbf{M}(C, \sigma-p.f)$: Every compact metrizable space has a $\sigma$-point-finite base.

\end{enumerate}
\end{definition}

The notation of type $\mathbf{M}(C, \square)$ was started in \cite{kk1} and continued in \cite{k} but not all forms from the definition above were defined in \cite{kk1} and \cite{k}. The forms $\mathbf{M}(C,L)$ and $\mathbf{M}(C, WO)$ were denoted by $CML$ and $CMWO$ in \cite{kert}. Most  forms from Definition \ref{s1d13} are new here. That the new forms $\mathbf{M}(C,\hookrightarrow \lbrack 0,1]^{\mathbb{N}})$, $\mathbf{M}(C,\hookrightarrow \lbrack 0,1]^{\mathbb{R}})$,  $\mathbf{M}(C, W(\mathbb{R}))$, $\mathbf{M}(C, B(\mathbb{R}))$,  $\mathbf{M}(C, |\mathcal{B}_Y|\leq |\mathcal{B}|)$ and  $\mathbf{M}([0,1], |\mathcal{B}_Y|\leq |\mathcal{B}|)$ are all important is shown in Section \ref{s4}.

Apart from the forms defined above, we also refer to the following forms that are not weaker than $\mathbf{AC}$ in $\mathbf{ZF}$:

\begin{definition}
\label{s1d14}
\begin{enumerate}
\item $\mathbf{LW}$ (\cite[Form 90]{hr}): For every linearly ordered set $\langle X, \leq\rangle$, the set $X$ is well-orderable.
\item $\mathbf{CH}$ (the Continuum Hypothesis): $2^{\aleph_0}=\aleph_1$.
\end{enumerate}
\end{definition}

\begin{remark}
\label{s1r15}
It is well known that $\mathbf{AC}$ and $\mathbf{LW}$ are equivalent in $\mathbf{ZF}$; however, $\mathbf{LW}$ does not imply $\mathbf{AC}$ in $\mathbf{ZFA}$ (see \cite{hr} and  \cite[Theorems 9.1 and 9.2]{j}).
\end{remark}

\section{Introduction}
\label{s2}

\subsection{The content of the article in brief}
\label{s2.1}

Although mathematicians are aware that a lot of theorems of $\mathbf{ZFC}$ that are included in standard textbooks on general topology (e.g., in \cite{En} and \cite{w}) may fail in $\mathbf{ZF}$ and many amazing disasters in topology in $\mathbf{ZF}$ have been discovered, new non-trivial results showing significant differences between truth values in $\mathbf{ZFC}$ and in $\mathbf{ZF}$ of some given propositions can be still surprising. In this article, we show new results concerning forms of type $\mathbf{M}(C, \square)$ in $\mathbf{ZF}$. The main aim of our work is to establish in $\mathbf{ZF}$ the set-theoretic strength of the forms of type $\mathbf{M}(C, \square)$, as well as relationships between these forms and relevant ones. Taking care of the readability of the article, in the forthcoming Subsections \ref{s2.2}--\ref{s2.4},  we include some known facts and few definitions for future references. In particular, in Subsection \ref{s2.4}, we give definitions of permutation models (called also Fraenkel-Mostowski models) and formulate only this version of a transfer theorem due to Pincus (called here the \emph{Pincus Transfer Theorem}) (cf. Theorem \ref{pinth}) which is applied in this article.  The main new results of the article are included in Sections \ref{s3}--\ref{s5}. Section \ref{s6} contains a list of open problems that suggest a direction for future research in this field.

In Section \ref{s3}, we construct new symmetric $\mathbf{ZF}$-models in each of which the conjunction $\mathbf{CH}\wedge\mathbf{WP}(\mathcal{P}(\mathbb{R}))\wedge\neg\mathbf{CAC}_{fin}$ is true.

 In Section \ref{s4}, we investigate relationships between the forms $\mathbf{M}(TB, WO)$, $\mathbf{M}(TB, S)$, $\mathbf{ICMDI}$ and $\mathbf{M}(C, S)$. Among other results of Section \ref{s4}, by showing appropriate permutation models and using the Pincus Transfer Theorem, we prove that the conjunctions $\mathbf{BPI}\wedge\mathbf{ICMDI}\wedge \neg\mathbf{IDI}$, $(\neg\mathbf{BPI})\wedge\mathbf{ICMDI}\wedge\neg\mathbf{IDI}$ and $\mathbf{UT}(\aleph_0,\aleph_0, cuf)\wedge\neg\mathbf{ICMDI}$ have $\mathbf{ZF}$-models (see Theorems \ref{s4t12},  \ref{s4t13} and \ref{s4t23}, respectively). We deduce that the conjunction $\mathbf{UT}(\aleph_0, \aleph_0, cuf)\wedge\neg\mathbf{M}(C, S)$ has a $\mathbf{ZF}$-model (see Corollary \ref{s4c24}). We discuss a relationship between $\mathbf{M}(C, S)$ and $\mathbf{CUC}$. Taking the opportunity, we fill in a gap in \cite{hr} and \cite{hr1} by proving that $\mathbf{WoAm}$ implies $\mathbf{CUC}$ (see Proposition \ref{s4p18}). 
 
Among a plethora of results of Section \ref{s5}, we show that $\mathbf{CAC}_{fin}$ implies neither $\mathbf{M}(C, S)$ nor $\mathbf{M}(C, \leq |\mathbb{R}|)$ (see Proposition \ref{s5p1}), and $\mathbf{M}(C, S)$ is equivalent to each one of the conjunctions: $\mathbf{CAC}_{fin}\wedge \mathbf{M}(C, \sigma-l.f.)$, $\mathbf{CAC}_{fin}\wedge\mathbf{M}(C, STB)$ and $\mathbf{CAC}(\mathbb{R}, C)\wedge\mathbf{M}(C, \leq|\mathbb{R}|)$ (see Theorems \ref{s5t02} and \ref{s5t2}, respectively).  We deduce that $\mathbf{M}(C, \sigma-l.f)$ is unprovable in $\mathbf{ZF}$ (see Remark \ref{s5r03}). We also deduce that $\mathbf{M}(C, S)$ and $\mathbf{M}(C, \leq|\mathbb{R}|)$ are equivalent in every permutation model (see Corollary \ref{s5c3}). Furthermore, we prove that $\mathbf{M}(C, S)$ and $\mathbf{M}(C, \hookrightarrow[0,1]^{\mathbb{N}})$ are equivalent (see Theorem \ref{s5t6}). We show that, surprisingly, $\mathbf{M}(C, |\mathcal{B}_Y|\leq|\mathcal{B}|)$ and $\mathbf{M}(C, B(\mathbb{R}))$ are independent of $\mathbf{ZF}$ (see Theorem \ref{s5t9}). In Theorem \ref{s5t10}, we show that $\mathbf{M}(C, B(\mathbb{R}))$ is equivalent to the conjunction $\mathbf{M}(C, \hookrightarrow [0,1]^{\mathbb{R}})\wedge\mathbf{Part}(\mathbb{R})$,  $\mathbf{CAC}(\mathbb{R})$ implies that $\mathbf{M}(C, S)$ and $\mathbf{M}(C, B(\mathbb{R}))$ are equivalent; moreover, the statement that $\mathbf{M}(C, S)$ and $\mathbf{M}(C, W(\mathbb{R}))$ are equivalent is equivalent to $\mathbf{M}(C, B(\mathbb{R}))$. The models of $\mathbf{ZF}+\mathbf{WO}(\mathcal{P}(\mathbb{R}))+\neg\mathbf{CAC}_{fin}$ constructed in Section \ref{s3} are applied to a proof that  $\mathbf{Part}(\mathbb{R})$ does not imply $\mathbf{M}(C, \hookrightarrow [0, 1]^{\mathbb{R}})$ in $\mathbf{ZF}$ (see Theorem \ref{s5t11}).

\subsection{A list of several known theorems}
\label{s2.2}

We list below some known theorems for future references.

\begin{theorem}
\label{s2t1} 
(Cf. \cite{nag}.) $\mathbf{CAC}$ implies $\mathbf{M}(TB, S)$ in $\mathbf{ZF}$.
\end{theorem}

\begin{theorem}
\label{s2t2}
(Cf. \cite{k}.)$(\mathbf{ZF})$
\begin{enumerate}
\item[(i)] Let $\mathbf{X}=\langle X, d\rangle$  be an
uncountable compact separable metric space. Then $|X|=|\mathbb{R}|$.
\item[(ii)]  $\mathbf{CAC}_{fin}$ follows from each of the statements: $\mathbf{M}(C, S),\mathbf{M}(C, \leq |\mathbb{R}%
|) $ and ``For every compact metric space $\langle X, d\rangle$,
either $|X|\leq $ $|\mathbb{R}|$ or $|\mathbb{R}|\leq |X|$''.
\end{enumerate}
\end{theorem}

\begin{theorem}
\label{s2t3}
$(\mathbf{ZF})$ 
\begin{enumerate}
\item[(a)] (\cite[Theorem 8]{kt}.) The statements $\mathbf{M}(C, S)$, $\mathbf{CAC}(C, M)$ are equivalent.
\item[(b)] (\cite[Corollary 1(a)]{kt}.) $\mathbf{CAC}(C,M)$ implies $\mathbf{CAC}_{fin}$.
\end{enumerate}
\end{theorem}

\begin{theorem}
\label{s2t4} (\cite[Corollary 4.8]{gt}, Urysohn's Metrization Theorem.) $(\mathbf{ZF})$ If $%
\mathbf{X}$ is a second-countable $T_3$-space, then $\mathbf{X}$ is
metrizable.
\end{theorem}

\begin{theorem}
\label{s2t5}
 (Cf. \cite{naga}, \cite{sm}, \cite{bi}, \cite{cr}. ) 
 \begin{enumerate}
 \item[(i)]  $(\mathbf{ZFC})$ Every metrizable space has a $\sigma$-locally finite base.
 \item[(ii)] $(\mathbf{ZF})$ If a $T_1$-space $\mathbf{X}$ is regular and has a $\sigma$-locally finite base, then $\mathbf{X}$ is metrizable.
 \end{enumerate}
\end{theorem}

\begin{remark}
\label{s2r6}
That it holds in $\mathbf{ZFC}$ that a $T_1$-space is metrizable if and only if it is regular and has a $\sigma$-locally finite base was originally proved  by Nagata in \cite{naga}, Smirnov in \cite{sm} and Bing in \cite{bi}. It was shown in \cite{cr} that it is provable in $\mathbf{ZF}$ that every regular $T_1$-space which admits a $\sigma$-locally finite base is metrizable. It was established in \cite{hkrs} that $\mathbf{M}(\sigma-l.f.)$ is an equivalent to $\mathbf{M}(\sigma-p.f.)$ and implies $\mathbf{MP}$. Using similar arguments, one can prove that $\mathbf{M}(C, \sigma-l.f.)$ and $\mathbf{M}(C, \sigma-p.f)$ are also equivalent in $\mathbf{ZF}$.
In \cite{gtw}, a model of $\mathbf{ZF}+\mathbf{DC}$ was shown in which $\mathbf{MP}$ fails. In \cite{sc}, a model of $\mathbf{ZF+BPI}$ was shown in which $\mathbf{MP}$ fails. This implies that, in each of the above-mentioned $\mathbf{ZF}$- models constructed in \cite{gtw} and \cite{sc},  there exists a metrizable space which fails to have a $\sigma$-point-finite base. This means that $\mathbf{M}(\sigma-l.f)$ is unprovable in $\mathbf{ZF}$. In Section \ref{s4}, it is clearly explained that $\mathbf{M}(C, \sigma-f.l)$ is also unprovable in $\mathbf{ZF}$. 
\end{remark}

\begin{theorem}
\label{s2t7} $(\mathbf{ZF})$ 
\begin{enumerate}
\item[(i)] (Cf. \cite{kert}.) A compact metrizable space is Loeb iff it is
second-countable iff it is separable. In consequence, the statement $\mathbf{M}(C, L)$, $\mathbf{M}(C, S)$ and $\mathbf{M}(C, 2)$ are all equivalent.
\item[(ii)] (Cf. \cite{kw}.) If $\mathbf{X}$ is a compact second-countable and metrizable
space, then $\mathbf{X}^{\omega }$ is compact and separable. In particular,
the Hilbert cube $[0,1]^{\mathbb{N}}$ is a compact, separable metrizable space.
\item[(iii)] (Cf. \cite{k}.) $\mathbf{BPI}$ implies $\mathbf{M}(C, S)$ and $\mathbf{M}%
(C,\leq |\mathbb{R}|)$.
\end{enumerate}
\end{theorem}

\subsection{Frequently used metrics}
\label{s2.3}

Similarly to \cite{kw1}, we make use of the following idea several times in the sequel. 

Suppose that  $\mathcal{A}=\{A_n: n\in\mathbb{N}\}$ is a disjoint family of non-empty sets, $A=\bigcup\limits_{n\in\mathbb{N}}A_n$ and $\infty\notin A$. Let $X=A\cup\{\infty\}$.  Suppose that $(\rho_n)_{n\in\mathbb{N}}$ is a sequence such that, for each $n\in\mathbb{N}$, $\rho_n$ is a metric on $A_n$. Let $d_n(x,y)=\min\{\rho_n(x,y), \frac{1}{n}\}$ for all $x,y\in A_n$. We define a function $d:X\times X\to\mathbb{R}$ as follows:

\[(\ast)\text{  } d(x,y)=\begin{cases} 0 &\text{if $x=y$;}\\
\max\{\frac{1}{n}, \frac{1}{m}\} &\text{if $x\in A_n ,y\in A_m$ and $n\neq m$;}\\
d_n(x,y) &\text{if  $x,y\in A_n$;}\\
\frac{1}{n} &\text{if $x\in A$ and $y=\infty$ or $x=\infty$ and $y\in A$.}\end{cases}
\]
\begin{proposition}
\label{s2p8}
The function $d$, defined by ($\ast$), has the following properties:
\begin{enumerate}
\item[(i)] $d$ is a metric on $X$ (cf. \cite{kw1});
\item[(ii)] if, for every $n\in\mathbb{N}$, the space $\langle A_n, \tau(\rho_n)\rangle$ is compact, then so is the space $\langle X, \tau(d)\rangle$ (cf. \cite{kw1});
\item[(iii)] the space $\langle X, \tau(d)\rangle$ has a $\sigma$-locally finite base;
\item[(iv)] if  $\mathcal{A}$ does not have a choice function, the space $\langle X, \tau(d)\rangle$ is not separable.
\end{enumerate}
\end{proposition}

Metrics defined by $(\ast)$ were used, for instance, in \cite{kerta}, \cite{kert}, \cite{kw1}, as well as in several other papers not cited here. 

\subsection{Permutation models and the Pincus Transfer Theorem}
\label{s2.4}

Let us clarify definitions of the permutation models we deal with. We refer to \cite[Chapter 4]{j} and \cite[Chapter 15, p. 251]{j1} for the basic terminology and facts concerning permutation models.

Suppose we are given a model $\mathcal{M}$ of $\mathbf{ZFA+AC}$ with an infinite set $A$ of all atoms of $\mathcal{M}$, and a group $\mathcal{G}$ of permutations of $A$. For a set $x\in\mathcal{M}$, we denote by $\TC(x)$ the transitive closure of $x$ in $\mathcal{M}$. Then every permutation $\phi$ of $A$ extends uniquely to an $\in$-automorphism (usually denoted also by $\phi$) of $\mathcal{M}$. For $x\in \mathcal{M}$, we put:
$$\fix_{\mathcal{G}}(x)=\{\phi\in\mathcal{G}: (\forall t\in x)\phi(t)=t\}\text{ and } \sym_{\mathcal{G}}(x)=\{\phi\in\mathcal{G}: \phi(x)=x\}.$$
We refer the readers to \cite[Chapter 4, pp. 46--47]{j} for the definitions of the concepts of a \emph{normal filter} and a \emph{normal ideal}. 

\begin{definition}
\label{s2d9}
\begin{enumerate}
\item[(i)] The \emph{permutation model} $\mathcal{N}$  \emph{determined by} $\mathcal{M}, \mathcal{G}$ \emph{and a normal filter}  $\mathcal{F}$ of subgroups of $\mathcal{G}$ is defined by the equality:
$$\mathcal{N}=\{x\in\mathcal{M}: (\forall t\in\TC(\{x\}))(\sym_{\mathcal{G}}(t)\in\mathcal{F})\}.$$
\item[(ii)] The \emph{permutation model} $\mathcal{N}$ \emph{determined by} $\mathcal{M}, \mathcal{G}$ \emph{and a normal ideal} $\mathcal{I}$ of subsets of the set of all atoms of $\mathcal{M}$ is defined by the equality:
$$\mathcal{N}=\{x\in\mathcal{M}: (\forall t\in\TC(\{x\}))(\exists E\in\mathcal{I}) (\fix_{\mathcal{G}}(E)\subseteq\sym_{\mathcal{G}}(t))\}.$$
\item[(iii)] (Cf. \cite[p. 46]{j} and \cite[p. 251]{j1}.)  A \emph{permutation model} (or, equivalently, a \emph{Fraenkel-Mostowski model}) is every class $\mathcal{N}$ which can be defined by (i).
\end{enumerate}
\end{definition}

\begin{remark}
\label{s2r10}
($a$) Let $\mathcal{F}$ be a normal filter of subgroups of $\mathcal{G}$ and let $x\in \mathcal{M}$. If $\text{sym}_{\mathcal{G}}(x)\in\mathcal{F}$, then $x$ is called \emph{symmetric}. If  every element of $\TC(\{x\})$ is symmetric, then $x$ is called \emph{hereditarily symmetric} (cf. \cite[p. 46]{j} and \cite[p. 251]{j1}). 

($b$) Given a normal ideal $\mathcal{I}$ of subsets of the set $A$ of atoms of $\mathcal{M}$, the filter $\mathcal{F}_{\mathcal{I}}$ of subgroups of $\mathcal{G}$ generated by $\{\fix_{\mathcal{G}}(E): E\in\mathcal{I}\}$ is a normal filter such that the permutation model determined by $\mathcal{M}, \mathcal{G}$ and $\mathcal{F}_{\mathcal{I}}$ coincides with the permutation model determined by $\mathcal{M}, \mathcal{G}$ and $\mathcal{I}$ (see \cite[p. 47]{j}). For $x\in\mathcal{M}$, a set $E\in\mathcal{I}$ such that $\fix_{\mathcal{G}}(E)\subseteq \sym_{\mathcal{G}}(x)$ is called a \emph{support} of $x$.
\end{remark}

In the forthcoming sections, we describe and apply several permutation models. For example, we apply the permutation model which appeared in \cite[the proof to Theorem 2.5]{kerta} and was also used in \cite{kert}, the \emph{Basic Fraenkel Model} (labeled as $\mathcal{N}1$ in \cite{hr}) and the \emph{Mostowski Linearly Ordered Model} (labeled as $\mathcal{N}3$ in \cite{hr}). Let us give definitions of these models and recall some of their properties for future references.

\begin{definition}
\label{s2d11}
(Cf. \cite{kerta}.) Let $\mathcal{M}$ be a model of $\mathbf{ZFA}+\mathbf{AC}$. Let $A$ be the set of all atoms  of $\mathcal{M}$ and let $\mathcal{I}=[A]^{<\omega}$. Assume that:
 \begin{enumerate}
 \item[(i)] $A$ is expressed as $\bigcup\limits_{n\in\mathbb{N}} A_{n}$ where $\{A_n: n\in\mathbb{N}\}$ is a disjoint family such that,  for every $n\in \mathbb{N}$, 
\[
A_{n}=\{a_{n,x}:x\in S(0,\frac{1}{n})\} 
\]%
and $S(0,\frac{1}{n})$ is the circle of the Euclidean plane $\langle \mathbb{R}^{2},\rho_e \rangle$ 
of radius $\frac{1}{n}$, centered at $0$;
\item[(ii)]  $\mathcal{G}$ is the group of all permutations of $A$ that rotate the $A_{n}$'s by an angle $%
\theta _{n}\in \mathbb{R}\ $.
\end{enumerate}
Then the permutation model $\mathcal{N}_{cr}$ determined by $\mathcal{M}, \mathcal{G}$ and the normal ideal $\mathcal{I}$ will be called the \emph{concentric circles permutation model}.
\end{definition}

\begin{remark}
\label{s2r12}
We need to recall some properties of $\mathcal{N}_{cr}$ for applications in this paper. Let us use the notation from Definition \ref{s2d11}. In \cite[the proof to Theorem 2.5]{kerta}, it was proved that $ \{A_{n}:n\in \mathbb{N}\}$ does not have a multiple choice function in $\mathcal{N}_{cr}$.  In \cite[the proof to Theorem 3.5]{kert}, it was proved that $\mathbf{IDI}$ holds in $\mathcal{N}_{cr}$, so $\mathbf{CAC}_{fin}$ also holds in $\mathcal{N}_{cr}$ (see Remark \ref{s1r10}(i)).
\end{remark}

\begin{definition}
\label{s2d13}
 (Cf. \cite[p. 176]{hr} and \cite[Section 4.3]{j}.) Let $\mathcal{M}$ be a model of $\mathbf{ZFA}+\mathbf{AC}$. Let $A$ be the set of all atoms  of $\mathcal{M}$ and let $\mathcal{I}=[A]^{<\omega}$. Assume that:
\begin{enumerate}
\item[(i)] $A$ is a denumerable set;
\item[(ii)] $\mathcal{G}$ is the group of all permutations of $A$.
\end{enumerate}
Then the \emph{Basic Fraenkel Model} $\mathcal{N}1$ is the permutation model determined by $\mathcal{M}$, $\mathcal{G}$ and $\mathcal{I}$. 
\end{definition}

\begin{remark}
\label{s2r14}
It is known that, in $\mathcal{N}1$, the set $A$ of all atoms is amorphous, so $\mathbf{IDI}$ fails (see  \cite[p. 52]{j} and  \cite[pp, 176--177]{hr}). It is also known that $\mathbf{BPI}$ is false in $\mathcal{N}1$ but $\mathbf{CAC}_{fin}$ is true in $\mathcal{N}1$ (see  \cite[p. 177]{hr}).
\end{remark}

\begin{definition}
\label{s2d15} 
(Cf. \cite[p. 182]{hr} and \cite[Section 4.6]{j}.) Let $\mathcal{M}$ be a model of $\mathbf{ZFA}+\mathbf{AC}$. Let $A$ be the set of all atoms  of $\mathcal{M}$ and let $\mathcal{I}=[A]^{<\omega}$.  Assume that:
\begin{enumerate}
\item[(i)] the set $A$ is denumerable and there is a fixed ordering $\leq$ in $A$ such that $\langle A, \leq\rangle$ is order isomorphic to the set of all rational numbers equipped with the standard linear order;
\item[(ii)] $\mathcal{G}$ is the group of all order-automorphisms of $\langle A, \leq\rangle$.
\end{enumerate}
Then the \emph{Mostowski Linearly Ordered Model} $\mathcal{N}3$ is the permutation model determined by $\mathcal{M}, \mathcal{G}$ and $\mathcal{I}$.
\end{definition}

\begin{remark}
\label{s2r16}
It is known that the power set of the set of all atoms is Dedekind-finite in $\mathcal{N}3$, so $\mathbf{IDI}$ fails in $\mathcal{N}3$ (see  \cite[pp. 182--183]{hr}). However, $\mathbf{BPI}$ and $\mathbf{CAC}_{fin}$ are true in $\mathcal{N}3$ (see \cite[p, 183]{hr}).
\end{remark}

 It is well known that, in any permutation model, the power set of any pure set (that is, a set with no atoms in its transitive closure) is well-orderable (see, e.g., \cite[p. 176]{hr}). This can be deduced from the following helpful proposition:
 
 \begin{proposition}
 \label{s2p17}
(Cf. \cite[Item (4.2), p. 47]{j}.) Let $\mathcal{N}$ be the permutation model determined by $\mathcal{M}, \mathcal{G}$ and a normal filter $\mathcal{F}$. For every $x\in\mathcal{N}$, $x$ is well-orderable in $\mathcal{N}$ iff $ \fix_{\mathcal{G}}(x)\in\mathcal{F}$.
 \end{proposition}
 
\begin{remark}
\label{s2r18} 
If a statement $\mathbf{A}$ is satisfied in a permutation model, to show that there exists a $\mathbf{ZF}$-model in which $\mathbf{A}$ is satisfied, we use transfer theorems due to Pincus (cf. \cite{pin} and \cite{pin1}). Pincus transfer theorems, together with definitions of a \emph{boundable formula} and an \emph{injectively boundable formula} that are involved in the theorems, are included in \cite[Note 103]{hr}.
\end{remark}
 To our transfer results, we apply mainly the following fragment of the third theorem from \cite[p. 286]{hr}:
 
 \begin{theorem}
 \label{pinth} 
(The Pincus Transfer Theorem.) (Cf. \cite{pin}, \cite{pin1} and \cite[p. 286]{hr}.)  Let $\mathbf{\Phi}$ be a conjunction of statements that are either injectively boundable or $\mathbf{BPI}$. If $\mathbf{\Phi}$ has a permutation model, then $\mathbf{\Phi}$ has a $\mathbf{ZF}$-model.
 \end{theorem}
 
 In the definition of an injectively boundable formula, a notion of an injective cardinality is involved. This notion is given, for instance, in \cite[Item (3), p. 284]{hr}. Let us formulate its equivalent definition below.

\begin{definition}
\label{s2d20} 
For a set $x$, the \emph{injective cardinality} of $x$ is the von Neumann cardinal number $|x|_{-}$ defined as follows:
$$|x|_{-}=\sup\{\kappa: \kappa \text{ is a von Neumann cardinal equipotent to a subset of } x\}.$$
\end{definition}

Now, we are in the position to pass to the main body of the article.

\section{New symmetric models}
\label{s3}

Suppose that $\Phi$ is a form that is satisfied in a $\mathbf{ZFA}$-model. Even if $\mathbf{\Phi}$ fulfills the assumptions of the Pincus Transfer Theorem, it might be complicated to check it and to see well a $\mathbf{ZF}$-model in which $\mathbf{\Phi}$ is satisfied. This is why it is good to give a direct relatively simple description of a $\mathbf{ZF}$-model satisfying $\mathbf{\Phi}$. In the proof to Theorem \ref{s3t1} below, we show a class of symmetric models satisfying $\mathbf{CH}\wedge\mathbf{WO}(\mathcal{P}(\mathbb{R}))\wedge\neg\mathbf{CAC}_{fin}$. In Section \ref{s5}, models of this class are applied to a proof that the conjunction $\mathbf{Part}(\mathbb{R})\wedge\neg \mathbf{M}(C, \hookrightarrow [0,1]^{\mathbb{R}})$ has a $\mathbf{ZF}$-model (see Theorem \ref{s5t11}).

\begin{theorem}
\label{s3t1}
Let $n,\ell\in\omega\setminus\{0,1\}$. There is a symmetric model $N_{n,\ell}$ of $\mathbf{ZF}$ such that 
$$N_{n,\ell}\models \forall m\in n(2^{\aleph_{m}}=\aleph_{m+1})\wedge\neg\mathbf{CAC}_{\ell}.$$
Hence, it is also the case that $$N_{n,\ell}\models \mathbf{CH}\wedge\mathbf{WO}(\mathcal{P}(\mathbb{R}))\wedge\neg\mathbf{CAC}_{fin}.$$ 
\end{theorem}

\begin{proof}
Let us use the terminology and results from \cite[Chapter VII]{ku1} and \cite[Chapter 5]{j}. By \cite[Theorem 6.18, p. 216]{ku1}, we can fix a countable transitive model $M$ of $\mathbf{ZFC}+\forall m\in n(2^{\aleph_{m}}=\aleph_{m+1})$.  Our plan is to construct a symmetric extension model $N_{n,\ell}$ of $M$ with the required properties.

Let $\mathbb{P}=\Fn(\omega\times\ell\times\omega_{n}\times\omega_{n},2,\omega_n)$ be the set of all partial functions $p$ with $|p|<\aleph_{n}$, $\dom(p)\subseteq\omega\times \ell\times\omega_{n}\times\omega_{n}$ and $\ran(p)\subseteq 2=\{0,1\}$, partially ordered by reverse inclusion, i.e., for $p,q\in\mathbb{P}$, $p\le q$ if and only if $p\supseteq q$. The poset $\langle\mathbb{P},\leq\rangle$ has the empty function as its maximum element, which we denote by $\mathbf{1}$. Furthermore, since $\omega_{n}$ is a regular cardinal, it follows from \cite[Lemma 6.13, p. 214]{ku1} that $\langle \mathbb{P},\le\rangle$ is an $\omega_{n}$-closed poset. Hence, by \cite[Theorem 6.14, p. 214]{ku1}, forcing with $\mathbb{P}$ adds no new subsets of $\omega_{m}$ for $m\in n$, and hence it adds no new reals or sets of reals, but it does add new subsets of $\omega_{n}$. Furthermore, by \cite[Corollary 6.15, p. 215]{ku1}, we have that $\mathbb{P}$ preserves cofinalities $\leq\omega_{n}$, and hence cardinals $\leq\omega_{n}$. 

Let $G$ be a $\mathbb{P}$-generic filter over $M$, and let $M[G]$ be the corresponding generic extension model of $M$. In view of the above, for every model $N$ with $M\subseteq N\subseteq M[G]$, we have the following: 
$$N\models\forall m\in n(2^{\aleph_{m}}=\aleph_{m+1}).$$
By \cite[Theorem 4.2, p. 201]{ku1}, $\mathbf{AC}$ is true  $M[G]$. 

In $M[G]$, for $k\in\omega$, $t\in\ell$, and $i\in\omega_{n}$, we define the following sets along with their canonical names:
\begin{enumerate}
\item $a_{k,t,i}=\{j\in\omega_{n}:\exists p\in G(p(k,t,i,j)=1)\}$,

$\overline{a_{k,t,i}}=\{\langle \check{j},p\rangle:j\in\omega_{n}\wedge p\in \mathbb{P}\wedge p(k,t,i,j)=1\}$.
\item $A_{k,t}=\{a_{k,t,i}:i\in\omega_{n}\}$,

$\overline{A_{k,t}}=\{\langle\overline{a_{k,t,i}},\mathbf{1}\rangle:i\in\omega_{n}\}$.
\item $A_k=\{A_{k,0},A_{k,1},\ldots,A_{k,(\ell-1)}\}$,

$\overline{A_k}=\{\langle \overline{A_{k,0}},\mathbf{1}\rangle,\langle\overline{A_{k,1}},\mathbf{1}\rangle,\ldots,\langle\overline{A_{k,(\ell-1)}},\mathbf{1}\rangle\}$.
\item $\mathcal A=\{A_{k}:k\in\omega\}$,

$\overline{\mathcal A}=\{\langle\overline{A_k},\mathbf{1}\rangle:k\in\omega\}$.
\end{enumerate}

Now, every permutation $\phi$ of $\omega\times\ell\times\omega_{n}$ induces an order-automorphism of $\langle\mathbb{P},\le\rangle$ by requiring, for every $p\in \mathbb{P}$, the following:
\begin{equation}\label{eq:2}
\begin{aligned}
\dom\phi(p)&=\{\langle\phi(k,t,i),j\rangle:\langle k,t,i,j\rangle\in\dom(p)\},\\
\phi(p)(\phi(k,t,i),j)&=p(k,t,i,j).
\end{aligned}
\end{equation}
Let $\mathcal G$ be the group of all order-automorphisms of $\langle\mathbb{P},\le\rangle$ induced (as in (\ref{eq:2})) by all those permutations $\phi$ of $\omega\times\ell\times\omega_{n}$ which are defined as follows.

 For every $k\in\omega$, let $\sigma_{k}$ be a permutation of $\ell=\{0,1,\ldots,\ell-1\}$ and also let $\eta_{k}$ be a permutation of $\omega_{n}$. We define 
\begin{equation}
\label{eq:a}
\phi(k,t,i)=\langle k,\sigma_{k}(t),\eta_{k}(i)\rangle,
\end{equation}
for all $\langle k,t,i\rangle \in\omega\times\ell\times\omega_{n}$.
By (\ref{eq:a}), it follows that for every $\phi\in\mathcal G$ such that $\phi(k,t,i)=\langle k,\sigma_{k}(t),\eta_{k}(i)\rangle$, we have that, for every $k\in\omega$ and every $t\in\ell$,  
\begin{equation}\label{eq1}\phi(\overline{A_{k,t}})=\overline{A_{k,\sigma_{k}(t)}},
\end{equation}
and thus, for every $k\in\omega$,
\begin{equation}\label{eq1-2}\phi(\overline{A_{k}})=\overline{A_k}.\end{equation}
It follows that for every $\phi\in\mathcal{G}$,
\begin{equation}
\label{eq1-3}\phi(\overline{\mathcal A})=\overline{\mathcal A}.\end{equation}

For every finite subset $E\subseteq\omega\times\ell\times\omega_{n}$, we let $\fix_{\mathcal G}(E)=\{\phi\in\mathcal G:{\forall e\in E}(\phi(e)=e)\}$ and we also let $\Gamma$ be the filter of subgroups of $\mathcal G$ generated by the filter base $\{\fix_{\mathcal G}(E):E\in[\omega\times \ell\times\omega_{n}]^{<\omega}\}$. Then $\Gamma$ is a normal filter on $\mathcal{G}$ (see \cite[Section 5.2, p. 64]{j} for the definition of the term ``normal filter''). An element $x\in M$ is called \emph{symmetric} if there exists a finite subset $E\subseteq\omega\times\ell\times\omega_{n}$ such that, for every $\phi\in\fix_{\mathcal G}(E)$, we have $\phi(x)=x$; if such a set $E$ exists, we call $E$ a \emph{support} of $x$. An element $x\in M$ is called \emph{hereditarily symmetric} if $x$ and all elements of the transitive closure of $x$ are symmetric. Let $\HS$ be the set of all hereditarily symmetric names in $M$. As in \cite[Definition 2.7, p. 189]{ku1}, for $\tau\in\HS$, let $\tau_{G}$ denote the value of the name $\tau$. Let $$N_{n,\ell}=\{\tau_{G}:\tau\in\HS\}$$
 be the symmetric extension model of $M$. Then  $N_{n,\ell}\subset M[G]$.

In view of the observations at the beginning of the proof, we have
$$N_{n,\ell}\models\forall m\in n(2^{\aleph_{m}}=\aleph_{m+1}),$$
and thus
$$N_{n,\ell}\models\mathbf{CH}\wedge\mathbf{WO}(\mathcal{P}(\mathbb{R})).$$

For $k\in\omega$, $t\in\ell$, and $i\in\omega_{n}$, the sets $a_{k,t,i}$, $A_{k,t}$, $A_k$, and $\mathcal A$ are all elements of $N_{n,\ell}$. Let us fix $k\in\omega$, $t\in\ell$, and $i\in\omega_{n}$. Then $E=\{\langle k,t,i\rangle\}$ is a support of $\overline{a_{k,t,i}}$ and $\overline{A_{k,t}}$.  By (\ref{eq1-2}) and (\ref{eq1-3}), we have that, for every $\phi\in\mathcal G$, $\phi(\overline{A_{k}})=\overline{A_k}$ and $\phi(\overline{\mathcal A})=\overline{\mathcal A}$. Thus, $a_{k,t,i}$, $A_{k,t}$, $A_k$, and $\mathcal A$ all belong to $N_{n,\ell}$. For $\sigma,\tau\in\HS$, let $\op(\sigma,\tau)$ be the name for the ordered pair $\langle\sigma_G,\tau_G\rangle$ (see \cite[Definition 2.16, p. 191]{ku1}). Let $f=\{\langle k, A_k\rangle: k\in\omega\}$ and $\dot{f}=\{\langle\op(\check{k},\overline{A_k}),\mathbf{1}\rangle:k\in\omega\}$. Since, for every $\phi\in\mathcal{G}$, $\phi(\dot{f})=\dot{f}$, we deduce that $\dot{f}$ is an $\HS$-name for the mapping $f$ (in $M[G]$). This proves that $\mathcal A$ is denumerable in $N_{n,\ell}$.

Now, by making suitable adjustments to the proof that $\mathbf{CAC}_{2}$ is false in the Second Cohen Model (see  \cite[Section 5.4, p. 68]{j}), one may verify that $\mathcal{A}$ has no partial choice function in the model $N_{n,\ell}$. We invite interested readers to fill in the missing details.
\end{proof}

\begin{remark}
\label{s3r2}
Let us note that Theorem \ref{s3t1} provides a class of symmetric models satisfying $\mathbf{CH}\wedge\mathbf{WO}(\mathcal{P}(\mathbb{R}))\wedge\neg\mathbf{CAC}_{fin}$.
\end{remark}

\section{Around $\mathbf{ICMDI}$ and $\mathbf{M}(TB, WO)$}
\label{s4}

Since every compact metric space is totally bounded and every infinite separable Hausdorff space is Dedekind-infinite, let us begin our investigations of the forms of type $\mathbf{M}(C, \square)$ with a deeper look at the forms $\mathbf{M}(TB, WO)$, $\mathbf{M}(TB, S)$ and $\mathbf{ICMDI}$. We include a simple proof to the following proposition for completeness. 

\begin{proposition}
\label{s4p1}
$(\mathbf{ZF})$ Let $\mathbf{X}=\langle X, d\rangle$ is a totally bounded metric space such that $X$ is well-orderable. Then $\mathbf{X}$ is separable.
\end{proposition}
\begin{proof}
Since $X$ is well-orderable, so is the set $Y=\bigcup\limits_{n\in\mathbb{N}}(X^{n}\times\{n\})$. Let $\leq$ be a fixed well-ordering in $Y$. For every $m\in\mathbb{N}$, let $y_m=\langle x_m, k_m\rangle\in X^{k_m}\times\{k_m\}$ be the first element of $\langle Y, \leq\rangle$ such that $X=\bigcup\{B_d(x_m(i),\frac{1}{m}): i\in k_m\}$. The set $D=\bigcup\limits_{m\in\mathbb{N}}\{x_m(i): i\in k_m\}$ is countable and dense in $\mathbf{X}$.
\end{proof}

\begin{theorem}
\label{s4t2}
$(\mathbf{ZF})$ 
\begin{enumerate}
\item[(i)] $\mathbf{M}(TB, WO)\rightarrow\mathbf{M}(TB, S)$ and $\mathbf{M}(C, WO)\rightarrow \mathbf{M}(C, S)$. None of these implications is reversible.
\item[(ii)] $\mathbf{M}(TB, WO)\rightarrow \mathbf{M}(C, WO)\rightarrow\mathbf{M}(C, S)\rightarrow\mathbf{ICMDI}$.
\item[(iii)] (Cf. \cite[Theorem 7 (i)]{k}.)  $\mathbf{CAC}\rightarrow \mathbf{M}(TB, S)\rightarrow\mathbf{M}(C, S)$. 
\item[(iv)] Neither $\mathbf{M}(TB, WO)$ nor $\mathbf{M}(TB, S)$ implies $\mathbf{CAC}$.
\end{enumerate}
\end{theorem}
\begin{proof}
It follows from Proposition \ref{s4p1} that the implications from (i) are both true. It is known that, in Feferman's model $\mathcal{M}2$ in \cite{hr}, $\mathbf{CAC}$ is true but $\mathbb{R}$ is not well-orderable (see \cite[p. 140]{hr}). Then $[0,1]$ is a compact, metrizable but not well-orderable space in $\mathcal{M}2$. Hence $\mathbf{M}(TB, S)\wedge\neg\mathbf{M}(TB, WO)$ and $\mathbf{M}(C, S)\wedge\neg\mathbf{M}(C, WO)$ are both true in $\mathcal{M}2$. This completes the proof to (i). In view of (i), it is obvious that (ii) holds. It is known from \cite{k} that (iii) also holds. It follows from the first implication of (i) that to prove (iv), it suffices to show that $\mathbf{M}(TB, WO)$ does not imply $\mathbf{CAC}$. 

It was shown in \cite[the proof to Theorem 15]{k} that there exists a model $\mathcal{M}$ of $\mathbf{ZF}+\neg\mathbf{CAC}$ in which it is true that if a metric space $\mathbf{X}=\langle X, d\rangle$ is sequentially bounded (i.e., every sequence of points of $\mathbf{X}$ has a Cauchy's subsequence), then $\mathbf{X}$ is well-orderable and separable. By  \cite[Theorem 7 (vii)]{k}, every totally bounded metric space is sequentially bounded. This shows that exists a model $\mathcal{M}$ of $\mathbf{ZF}$ in which $\mathbf{M}(TB, WO)\wedge\neg\mathbf{CAC}$ is true. Hence (iv) holds. 
\end{proof}
 That $\mathbf{ICMDI}$ does not imply $\mathbf{M}(C, S)$ is shown in Proposition \ref{s5p1}(iv). It is unknown whether $\mathbf{M}(C, WO)$ is equivalent to or weaker than $\mathbf{M}(TB, WO)$ in $\mathbf{ZF}$.
 
  To compare $\mathbf{M}(C, S)$ with $\mathbf{M}(TB, S)$, we recall that it was proved in \cite{kyri} that the implication $\mathbf{M}(TB, S)\rightarrow \mathbf{CAC}(\mathbb{R})$ holds in $\mathbf{ZF}$; however, the implication $\mathbf{CAC}\rightarrow \mathbf{M}(TB, S)$ of Theorem \ref{s2t1} is not reversible in $\mathbf{ZF}$. On the other hand, it is known that $\mathbf{CAC}(\mathbb{R})$ and $\mathbf{M}(C,S)$ are independent of each other in $\mathbf{ZF}$ (see, e.g., \cite{k}). The  following proposition, together with the fact that $\mathbf{M}(TB, S)$ implies $\mathbf{M}(C, S)$, shows that $\mathbf{M}(TB, S)$ is essentially stronger than $\mathbf{M}(C, S)$ in $\mathbf{ZF}$.

\begin{proposition}
\label{s4p03} $(\mathbf{ZF})$
\begin{enumerate}
\item[(i)] (Cf. \cite[Proposition 8]{kyri}.) If every totally bounded metric space is strongly totally bounded, then $\mathbf{CAC}(\mathbb{R})$ holds.
\item[(ii)] In Cohen's Original Model  $\mathcal{M}$1 of \cite{hr} the following hold: $\mathbf{M}(C, S)$ is true, $\mathbf{M}(TB, S)$ is false and there exists a totally bounded metric space which is not strongly totally bounded.
\item[(iii)] $\mathbf{M}(C, S)$ does not imply $\mathbf{M}(TB, S)$. 
\end{enumerate}
\end{proposition}
\begin{proof}
That (i) holds was proved in \cite{kyri}. It is known that $\mathbf{BPI}$ is true $\mathcal{M}1$ (see \cite[p. 147]{hr}). It follows from Theorem \ref{s2t7}(iii) that $\mathbf{M}(C, S)$ holds in $\mathcal{M}1$. On the other hand, it is known that $\mathbf{CAC}(\mathbb{R})$ fails in $\mathcal{M}1$ (see  \cite[p. 147]{hr}). Therefore, by (i), it holds in $\mathcal{M}1$ that there exists a totally bounded metric space which is not strongly totally bounded.
\end{proof}

We  recall that a topological space $\mathbf{X}$ is called \emph{limit point compact} if every infinite subset of $\mathbf{X}$ has an accumulation point in $\mathbf{X}$ (see, e.g., \cite{k}).

\begin{proposition} 
\label{s4p3}
$(\mathbf{ZFA})$ $\mathbf{WoAm}$ implies both $\mathbf{M}(TB, WO)$ and ``every limit point compact,  first-countable  $T_1$-space is well-orderable''. 
\end{proposition}
\begin{proof} 
Let us assume $\mathbf{WoAm}$. Consider an arbitrary metric space $\langle X, d\rangle$. Suppose that $X$ is not well-orderable. By $\mathbf{WoAm}$, there exists an amorphous subset $B$ of $X$. Let $\rho=d\upharpoonright B\times B$. Lemma 1 of \cite{hdkr} states that every metric on an amorphous set has a finite range. Therefore, the set $\ran(\rho)=\{\rho(x,y): x,y\in B\}$ is finite. Since $B$ is infinite, the set $\ran(\rho)\setminus\{0\}$ is non-empty. If $\varepsilon=\min(\ran(\rho)\setminus\{0\})$, then there does not exist a finite $\varepsilon$-net in $\langle B,\rho\rangle$ because $B$ is infinite and, for every $x\in B$, $B_{\rho}(x, \varepsilon)=\{x\}$. This implies that $\rho$ is not totally bounded. Hence $d$ is not totally bounded.

Now, suppose that $\mathbf{Y}=\langle Y, \tau\rangle$ is a first-countable, limit point compact $T_1$-space. Let $C$ be an infinite subset of $Y$. Since $\mathbf{Y}$ is limit point compact, the set $C$ has an accumulation point in $\mathbf{Y}$. Let $y_0$ be an accumulation point of $C$ and let $\{U_n: n\in\mathbb{N}\}$ be a base of neighborhoods of $y_0$ in $\mathbf{Y}$. Since $\mathbf{Y}$ is a $T_1$-space, we can inductively define an increasing sequence $(n_k)_{k\in\mathbb{N}}$ of natural numbers such that, for every $k\in\mathbb{N}$, $C\cap(U_{n_k}\setminus U_{n_{k+1}})\neq\emptyset$. This implies that $C$ is not amorphous. Hence, no infinite subset of $Y$ is amorphous, so $Y$ is well-orderable by $\mathbf{WoAm}$.
\end{proof}

\begin{corollary}
\label{s4c4} 
$\mathcal{N}1\models \mathbf{M}(TB, WO)$.  
\end{corollary}
\begin{proof}
This follows from Proposition \ref{s4p3} and the known fact that $\mathbf{WoAm}$ is true in $\mathcal{N}1$ (see p. 177 in \cite{hr}).
\end{proof}

To prove that $\mathbf{M}(TB, WO)$ does not imply $\mathbf{WoAm}$ in $\mathbf{ZFA}$, let us use the model $\mathcal{N}3$. 
In what follows, the notation concerning $\mathcal{N}3$ is the same as in Definition \ref{s2d15}. For $a,b\in A$ with $a<b$ (where $A$ is the set of atoms of $\mathcal{N}3$ and $\leq$ is the fixed linear order on $A$), we denote by $(a, b)$ the open interval in the linearly ordered set $\langle A, \leq)$; that is, $(a, b)=\{x\in A: a<x<b\}$. A proof to the following lemma can be found in \cite{hst}.

\begin{lemma}
\label{s4l5}
(Cf. \cite[Lemma 3.17 and its proof]{hst}.) If $X\in \mathcal N3$, $E$ is a support of $X$ and there is $x \in X$ for which $E$ is not a support, then there exist a subset $Y$ of $X$ and atoms $a, b \in A$ with $a < b$,  such that $Y\in \mathcal N3$, $E\cap (a, b)=\emptyset$ and, in $\mathcal N3$,  there exists a bijection $f:(a,b)\to Y$ having a support $E^{\prime}$ such that $E\cup\{a, b\}\subseteq E^{\prime}$ and $E^{\prime}\cap (a,b)=\emptyset$. 
\end{lemma}

\begin{theorem}
\label{s4t6} 
$\mathcal{N}3\models \mathbf{M}(TB, WO)$.
\end{theorem}
\begin{proof}
We use the notation from Definition \ref{s2d15}. Suppose that $\langle X, d\rangle$ is a metric space in $\mathcal{N}3$ such that $X$ is not well-orderable in $\mathcal{N}3$. Then $X$ is infinite. Let $E\in [A]^{<\omega}$ be a support of both $X$ and $d$. By Proposition \ref{s2p17}, there exists $x\in X$ such that $E$ is not a support of $x$. 

By Lemma \ref{s4l5}, there exist $a,b\in A$ with $a<b$ and  $(a,b)\cap E=\emptyset$, such that there exists in $\mathcal{N}3$ an injection $f:(a, b)\to X$ which has a support $E^{\prime}$ such that $E\cup\{a,b\}\subseteq E^{\prime}$ and $E^{\prime}\cap (a, b)=\emptyset$. We put $B=(a, b)$ and $\rho(x,y)=d(f(x), f(y))$ for all $x,y\in B$.  Let us notice that $\rho\in\mathcal{N}3$ because $E^{\prime}$ is also a support of $\rho$.  We prove that $\ran({\rho})=\{\rho(x,y): x,y\in B\}$ is a two-element set. To this aim, we fix $b_1, b_2\in B$ with $b_1<b_2$ and put $r=\rho(b_1, b_2)$. Let $u,v\in B$ and $u\neq v$. To show that $\rho(u,v)=r$, we must consider several cases regarding the ordering of the elements $b_1, b_2, u, v$. We consider only one of the possible cases since all the other cases can be treated in much the same way as the chosen one. So, assume, for example, that $b_2<u<v$. Let $\phi$ be an order-automorphism of $\langle A, \leq\rangle$ such that $\phi(b_1)=u$, $\phi(b_2)=v$, and $\phi$ is the identity mapping on $A\setminus B$. Then $\phi\in\text{fix}_{\mathcal{G}}(E^{\prime})$, so $\phi(\rho)=\rho$. This implies that $\phi(r)=\rho(\phi(b_1), \phi(b_2))$.  Since, in addition $\phi(r)=r$, we have $r=\rho(b_1,b_2)=\rho(\phi(b_1), \phi(b_2))=\rho(u,v)$. Therefore, $\ran(\rho)=\{0, r\}$. Since the range of $\rho$ is finite, in much the same way, as in the proof to Proposition \ref{s4p3}, we deduce that $\langle B, \rho\rangle$ is not totally bounded. Hence $d$ is not totally bounded. 
\end{proof}

\begin{corollary}
\label{s4c7} 
$\mathbf{M}(TB, WO)$ does not imply $\mathbf{WoAm}$ in $\mathbf{ZFA}$.
\end{corollary}
\begin{proof}
It is known that $\mathbf{WoAm}$ is false in $\mathcal{N}3$ (see \cite[p.183]{hr}). Therefore, the conjunction $\mathbf{M}(TB, WO)\wedge\neg\mathbf{WoAm}$ is true in $\mathcal{N}3$ by Theorem \ref{s4t6}.
\end{proof}

In contrast to Corollary \ref{s4c4} and Theorem \ref{s4t6}, we have the following proposition:

\begin{proposition}
\label{s4t8}
$\mathcal{N}_{cr}\models \neg\mathbf{M}(C, WO)$. 
\end{proposition}
\begin{proof}
It was shown in \cite[the proof to Theorem 2.5]{kerta} that, in $\mathcal{N}_{cr}$, there exists a compact metric space $\mathbf{X}=\langle X, d\rangle$ which is not weakly Loeb. Then $X$ cannot be well-orderable in $\mathcal{N}_{cr}$.
\end{proof}

\begin{remark}
\label{s4r9}
In \cite{et}, a symmetric  model $\mathcal{M}$ of $\mathbf{ZF}$ was constructed such that, in $\mathcal{M}$, there exists a compact metric space $\langle X, d\rangle$ which is not weakly Loeb; thus, $\mathbf{M}(C,WO)$ fails in $\mathcal{M}$. 
\end{remark}

It is obvious that $\mathbf{IDI}$ implies $\mathbf{ICMDI}$ in $\mathbf{ZF}$; however, it seems to be still an open problem of whether  this implication is not reversible in $\mathbf{ZF}$. To solve this problem, first of all, let us notice that the following corollary follows directly from  Corollary \ref{s4c4} and Theorem \ref{s4t6}:  

\begin{corollary}
\label{s4c10}
\begin{enumerate}
\item[(i)] $\mathcal{N}1\models (\mathbf{ICMDI}\wedge\neg\mathbf{IDI})$.
\item[(ii)] $\mathcal{N}3\models (\mathbf{BPI}\wedge\mathbf{ICMDI}\wedge\neg\mathbf{IDI})$.
\end{enumerate}
\end{corollary}
 To transfer $\mathbf{BPI}\wedge\mathbf{ICMDI}\wedge\neg\mathbf{IDI}$ to a model of $\mathbf{ZF}$, let us prove the following lemma:
 
\begin{lemma}
\label{s4l11}
$\mathbf{ICMDI}$ is injectively boundable. 
\end{lemma}
\begin{proof}  First, we put
$$\Phi(x)=\neg(\exists y)(y\subseteq x\wedge |y|=\omega),$$
and 
$$\Psi(x)=\text{``$x$ is finite''}.$$
Note that the formula $\Psi(x)$ is boundable (see \cite{pin} or  Note 103, p. 284 in \cite{hr}).  Now it is not hard to verify that $\mathbf{ICMDI}$ is logically equivalent to the statement $\Omega$, where
\begin{multline*}\Omega=(\forall x)(|x|_{-}\leq\omega\rightarrow(\forall \rho\in\mathcal{P}(x\times x\times\mathbb{R}))\\
[(\Phi(x)\wedge(\rho\text{ is a metric on $x$ such that $\langle X, \rho\rangle$ is compact}))
\rightarrow\Psi(x)])\end{multline*}
Since $\Omega$ is obviously injectively boundable, so is $\mathbf{ICMDI}$.
\end{proof}

\begin{theorem}
\label{s4t12}
The conjunction $\mathbf{BPI}\wedge\mathbf{ICMDI}\wedge\neg\mathbf{IDI}$ has a $\mathbf{ZF}$-model.
\end{theorem}
\begin{proof}
It is known that $\neg\mathbf{IDI}$ is boundable, and hence injectively boundable (see \cite[2A5, p. 772]{pin} or \cite[p. 285]{hr}).  By Theorem \ref{pinth}, if a conjunction of $\mathbf{BPI}$ with injectively boundable statements has a Fraenkel-Mostowski model, then it has a $\mathbf{ZF}$-model. This, together with Corollary \ref{s4c10}(ii) and Lemma \ref{s4l11}, completes the proof.
\end{proof}

\begin{theorem} 
\label{s4t13} 
The conjunction $(\neg\mathbf{BPI})\wedge\mathbf{ICMDI}\wedge\neg\mathbf{IDI}$ has a $\mathbf{ZF}$-model.
\end{theorem}
\begin{proof} Let $\mathbf{\Phi}$ be the conjunction $(\neg \mathbf{BPI})\wedge\mathbf{ICMDI}\wedge\neg\mathbf{IDI}$. It is known that $\mathbf{BPI}$ is false in $\mathcal{N}1$ (see \cite[p. 177]{hr}). Hence $\mathbf{\Phi}$ has a permutation model by Corollary \ref{s4c10}(i). Since the statements $\neg \mathbf{BPI}$, $\mathbf{ICMDI}$ and $\neg\mathbf{IDI}$ are all injectively boundable, $\mathbf{\Phi}$ has a $\mathbf{ZF}$-model by Theorem \ref{pinth}.
\end{proof}

\begin{corollary}
\label{s4c14}
$\mathbf{ICMDI}$ does not imply $\mathbf{IDI}$ in $\mathbf{ZF}$. Furthermore, $\mathbf{BPI}$ is independent of $\mathbf{ZF}+\mathbf{ICMDI}+\neg\mathbf{IDI}$. 
\end{corollary}

\begin{proposition}
\label{s4p15} 
$\mathbf{ICMDI}$ implies $\mathbf{CAC}_{fin}$ in $\mathbf{ZF}$.
\end{proposition}
\begin{proof}
It suffices to apply Corollary 2.2 (iii) in \cite{kert} which states that if every infinite compact metrizable space has an infinite well-orderable subset, then $\mathbf{CAC}_{fin}$ holds. 
\end{proof}

At this moment, it is unknown whether there is a model of $\mathbf{ZF}$ in which the conjunction $\mathbf{CAC}_{fin}\wedge\neg\mathbf{ICMDI}$ is false.

\begin{remark}
\label{s4r17}
In Cohen's original model $\mathcal{M}1$ in \cite{hr}, $\mathbf{CUC}$ holds and there exists a dense Dedekind-finite subset $X$ of the interval [0, 1] of $\mathbb{R}$. The metric space $\mathbf{X}=\langle X, d\rangle$, where $d(x,y)=|x-y|$ for all $x,y\in X$, is totally bounded but $X$ is not well-orderable in $\mathcal{M}1$. It was remarked in \cite{kyri} that, since $\langle X, d\rangle$ is not separable, $\mathbf{CUC}$ does not imply $\mathbf{M}(TB, S)$ in $\mathbf{ZF}$. Now, it is easily seen that $\mathbf{CUC}$ does not imply $\mathbf{M}(TB, WO)$ in $\mathbf{ZF}$.
\end{remark}

It is not stated in \cite{hr} nor in \cite{hr1} that $\mathbf{WoAm}$ implies $\mathbf{CUC}$. Since we have not seen a solution of the problem of whether this implication is true in other sources, let us notice that it follows from the following proposition that this implication holds in $\mathbf{ZF}$:

\begin{proposition}
\label{s4p18}
\begin{enumerate}
\item[(i)] $(\mathbf{ZFA})$ $\mathbf{WoAm}\rightarrow \mathbf{CUC}$;
\item[(ii)] If $\mathcal{N}$ is a model of $\mathbf{ZFA}$ in which $\mathbb{R}$ is well-orderable (in particular, if $\mathcal{N}$ is a permutation model), then:
 $$\mathcal{N}\models (\mathbf{CAC}_{WO}\rightarrow \mathbf{CUC}).$$ 
 \item[(iii)] If $\mathcal{N}$ is a permutation model, then: 
 $$\mathcal{N}\models (\mathbf{AC}_{fin}\rightarrow\mathbf{CUC}).$$
\end{enumerate}
\end{proposition}
\begin{proof}
Let $\mathcal{A}=\{ A_n : n\in\omega\}$ be a disjoint family of non-empty countable sets and let $A=\bigcup\limits_{n\in\omega}A_n$. Clearly, if $\bigcup\limits_{n\in\omega}(A_n\times\omega)$ is countable, then $A$ is countable. Therefore, to show that $A$ is countable, we may assume that, for every $n\in\omega$, the set $A_n$ is denumerable.  For every $n\in\omega$, let $B_n$ be the set of all bijections from $\omega$ onto $A_n$. Since $|\omega^{\omega}|=|\mathbb{R}^{\omega}|=|\mathbb{R}|$ and the sets $A_n$ are all denumerable, for every $n\in\omega$,  the set $B_n$ is equipotent to $\mathbb{R}$.
 
 (i)  Let $B=\bigcup\limits_{n\in\omega}B_n$. If $B$ is well-orderable, then there exists a sequence $(f_n)_{n\in\omega}$ of bijections $f_n:\omega\to A_n$, so  $A$ is countable. Suppose that $B$ is not well-orderable. Then it follows from $\mathbf{WoAm}$ that there exists an amorphous subset $C$ of $B$. Since $C$ cannot be partitioned into two infinite subsets,  the set $\{n\in\omega: C\cap B_n\neq\emptyset\}$ is finite. This implies that there exists $m\in\omega$ such that $C\subseteq\bigcup\limits_{n\in m+1} B_n$, so $C$ is equipotent to a subset of $\mathbb{R}$. But this is impossible because $\mathbb{R}$ does not have amorphous subsets. The contradiction obtained completes the proof to (i).\smallskip
 
 (ii) Now, assume that $\mathbb{R}$ is well-orderable and $\mathbf{CAC}_{WO}$ holds.  Then the sets $B_n$, being equipotent to $\mathbb{R}$, are all well-orderable. Hence, it follows from $\mathbf{CAC}_{WO}$ that there exists $f\in\prod\limits_{n\in\omega} B_n$. Then we have a sequence $(f(n))_{n\in\omega}$ of bijections $f(n):\omega\to A_n$, so $A$ is countable. 
 
 (iii) Since $\mathbf{AC}_{fin}$ implies $\mathbf{AC}_{WO}$ in every permutation model (cf. \cite{how1} and Note 2 in \cite{hr}), we infer that if $\mathbf{AC}_{fin}$ is satisfied in $\mathcal{N}$ and $\mathcal{N}$ is a permutation model, then  $\mathbf{CAC}_{WO}$ holds in $\mathcal{N}$. This, taken together with (ii),  implies (iii). 
\end{proof}

\begin{remark}
\label{s4r19}
($a$) In Felgner's Model I (labeled as $\mathcal{M}20$ in \cite{hr}), $\mathbf{AC}_{WO}$ holds and $\mathbf{CUC}$ fails (cf. \cite[p. 159]{hr}). We recall that $\mathbf{AC}_{fin}\wedge\neg\mathbf{AC}_{WO}$ is true in Sageev's Model I (labeled as model $\mathcal{M}6$ in \cite{hr}).  Moreover, since $\mathbf{IDI}\wedge\neg\mathbf{CUC}$ is true in $\mathcal{M}6$ (cf. \cite[p. 152]{hr}), $\mathbf{ICMDI}$ does not imply $\mathbf{CUC}$ is $\mathbf{ZF}$. 

($b$)  One should not claim that $\mathbf{CAC}_{fin}$ and $\mathbf{CAC}_{WO}$ are equivalent in every permutation model. Namely, let $\mathcal{N}$ be the permutation model of $\mathbf{IDI}\wedge \neg \mathbf{CUC}$ constructed in  \cite[the proof to Theorem 4 (iv)]{tach1}. Since $\mathbf{IDI}$ implies $\mathbf{CAC}_{fin}$, it follows  from Proposition \ref{s4p18} that, in this model  $\mathcal{N}$, $\mathbf{CAC}_{fin}$ is true but $\mathbf{CAC}_{WO}$ is false. 
\end{remark}

\begin{remark}
\label{s4r20}
($a$) It is unknown whether $\mathbf{M}(C, S)$ implies $\mathbf{CUC}$ in $\mathbf{ZF}$ or in $\mathbf{ZFA}$. We recall that $\mathbf{BPI}$ implies $\mathbf{M}(C, S)$. The problem of whether $\mathbf{BPI}$ implies $\mathbf{CUC}$ in $\mathbf{ZF}$ or in $\mathbf{ZFA}$ is still unsolved. However, if $\mathcal{N}$ is a permutation model in which $\mathbf{BPI}$ is true, then $\mathbf{AC}_{fin}$ is also true in $\mathcal{N}$ (see, e.g., \cite[Proposition 4.39]{hh}); hence, by Proposition \ref{s4p18}(iii), $\mathbf{BPI}$ implies $\mathbf{CUC}$ in every permutation model. 

($b$) Since $\mathbf{M}(C, S)$ implies $\mathbf{CAC}_{fin}$ (see Theorem \ref{s2t2}(ii) or Theorem \ref{s4t2} with Proposition \ref{s4p15}), it follows from Proposition \ref{s4p18}(iii) that $\mathbf{CAC}_{fin}\wedge\neg\mathbf{AC}_{fin}$ is satisfied in every permutation model of $\mathbf{M}(C, S)\wedge\neg\mathbf{CUC}$. 
\end{remark}

It still eludes us whether or not $\mathbf{CUC}$ implies $\mathbf{M}(C, S)$ in $\mathbf{ZF}$. However, we are able to provide a partial solution to this intriguing open problem by proving that the statement $\mathbf{UT}(\aleph_0,\aleph_0, cuf)\wedge\neg\mathbf{ICMDI}$ is a conjunction of injectively boundable statement and it  has a permutation model. It is obvious that $\neg\mathbf{ICMDI}$ is boundable, so also injectively boundable. To show that $\mathbf{UT}(\aleph_0, \aleph_0, cuf)$ is injectively boundable, we need the following lemma proved in \cite{hs}:

\begin{lemma}
\label{s4l21} 
$(\mathbf{ZF})$
(Cf. \cite[Lemma 3.5]{hs}.) For any ordinal $\alpha$, if $\mathcal{R}$ is a collection of sets such that $|\mathcal{R}|\leq\aleph_{\alpha+1}$ and, for every $x\in\mathcal{R}$, $|x|\leq\aleph_{\alpha}$, then $|\bigcup\mathcal{R}|\not\geq\aleph_{\alpha +2}$.
\end{lemma}

\begin{proposition}
\label{s4p22}
The statement $\mathbf{UT}(\aleph_0, \aleph_0, cuf)$ is injectively boundable.
\end{proposition}
\begin{proof}
In the light of Lemma \ref{s4l21},  $\mathbf{UT}(\aleph_{0},\aleph_{0},cuf)$ is  equivalent to the statement:  
\begin{multline}
\label{mult:hs}
(\forall x)(|x|\not\geq\aleph_{3}\rightarrow (\forall y)\text{ ``if $y$ is a countable collection of countable sets}\\ \text{whose union is $x$, then $x$ is a cuf set''}).
\end{multline}
Since, for every set $x$, the statements  $|x|\not\geq\aleph_3$ and $|x|_{-}\leq\aleph_{2}$ are equivalent, it is obvious that  (\ref{mult:hs}) is injectively boundable.  Thus, $\mathbf{UT}(\aleph_{0},\aleph_{0},cuf)$ is also injectively boundable.
\end{proof}

\begin{theorem}
\label{s4t23}
\begin{enumerate}
\item[(i)] The statement $\mathbf{LW}\wedge\mathbf{UT}(\aleph_{0},\aleph_{0},cuf)\wedge\neg \mathbf{ICMDI}$ has a permutation model. 

\item[(ii)] The statement $\mathbf{UT}(\aleph_{0},\aleph_{0},cuf)\wedge\neg\mathbf{ICMDI}$ has a $\mathbf{ZF}$-model.
\end{enumerate}
\end{theorem}
\begin{proof}
(i) Let us apply the permutation model $\mathcal{N}$ which was constructed in \cite[the proof to Theorem 3.3]{hdhkr}. To describe $\mathcal{N}$, we start with a model $\mathcal{M}$ of $\mathbf{ZFA}+\mathbf{AC}$ with a set $A$ of atoms such that $A$ has a denumerable partition $\{A_{i}:i\in\omega\}$ into denumerable sets, and for each $i\in\omega$, $A_{i}$ has a denumerable partition $P_{i}=\{A_{i,j}:j\in\mathbb{N}\}$ into finite sets such that, for every $j\in\omega\setminus\{0\}$, $|A_{i,j}|=j$.  Let $\text{Sym}(A)$ be the group of all permutations of $A$ and let
$$\mathcal{G}=\{\phi\in\text{Sym}(A): (\forall i\in\omega)(\phi(A_i)=A_i)\text{ and } |\{x\in A: \phi(x)\neq x\}|<\aleph_0\}.$$
 
Let $\mathbf{P}_{i}=\{\phi(P_{i}):\phi\in G\}$ and also let $\mathbf{P}=\bigcup\{\mathbf{P}_{i}:i\in\omega\}$. Let $\mathcal{F}$ be the normal filter of subgroups of $\mathcal{G}$ generated by the collection $\{\fix_{\mathcal{G}}(E): E\in [\mathbf{P}]^{<\omega}\}$. Then $\mathcal{N}$ is the permutation model determined by $\mathcal{M}$, $\mathcal{G}$ and $\mathcal{F}$. We say that a finite subset $E$ of $\mathbf{P}$ is a \emph{support} of an element $x$ of $\mathcal{N}$  if, for every $\phi\in\fix_{\mathcal{G}}(E)$, $\phi(x)=x$.

It was noticed in \cite[the proof to Theorem 3.3]{hdhkr} that, for every $i\in\omega$ and every $Q\in\mathbf{P}_i$,  the following hold:
\begin{enumerate}
\item[(a)] for any $\phi\in G$, $\phi$ fixes $Q$ if and only if $\phi$ fixes $Q$ pointwise;
\item[(b)] $(\exists j_{Q}\in\omega)(Q\supseteq\{A_{i,j}:j>j_{Q}\})$.
\end{enumerate}

To prove that $\mathbf{LW}$ is true in $\mathcal{N}$, we fix a linearly ordered set $\langle Y, \leq\rangle$ in $\mathcal{N}$ and prove that $\fix_{\mathcal{G}}(Y)\in\mathcal{F}$. To this aim, we choose a set $E\in [\mathbf{P}]^{<\omega}$ such that $E$ is a support of both $Y$ and $\leq$. To show that $\fix_{\mathcal{G}}(E)\subseteq\fix_{\mathcal{G}}(Y)$, let us consider any element $y\in Y$ and a permutation  $\phi\in\fix_{\mathcal{G}}(E)$. Suppose that $\phi(y)\neq y$. Then either $y<\phi(y)$ or $\phi(y)<y$. Since every element of $\mathcal{G}$ moves only finitely many atoms, there exists $k\in\omega\setminus\{0\}$ such that $\phi^{k}$ is the indentity mapping on $A$.  Assuming that $\phi(y)<y$, for such a $k$, we obtain the following: 
$$y<\phi(y)<\phi^{2}(y)<\ldots <\phi^{k-1}(y)<\phi^{k}(y)=y,$$
and thus $y<y$.  Arguing similarly, we deduce that if $\phi(y)<y$, then $y<y$. The contradiction obtained shows that $\phi(y)=y$ for every $y\in Y$ and every $\phi\in\fix_{\mathcal{G}}(E)$. Hence $\fix_{\mathcal{G}}(E)\subseteq\text{fix}_{\mathcal{G}}(Y)$. Since $\mathcal{F}$ is a filter and $\fix_{\mathcal{G}}(E)\in\mathcal{F}$, we infer that $\text{fix}_{\mathcal{G}}(Y)\in\mathcal{F}$. This, together with Proposition \ref{s2p17}, implies that the set $Y$ is well-orderable in $\mathcal{N}$. Hence $\mathcal{N}\models \mathbf{LW}$. 

Now, let us prove that $\mathbf{ICMDI}$ fails in $\mathcal{N}$. First, to find a metric $d$ on $A_0$ such that  $\langle A_0, d\rangle$ is a compact metric space in $\mathcal{N}$, we denote by $\infty$ the unique element of $A_{0,1}$ and, for every $n\in\mathbb{N}$, we denote by $\rho_n$ the discrete metric on $A_{0, n+1}$.  Then, making obvious adjustments in notation, we let $d$ be the metric on $A_0$ defined by $(\ast)$ in Subsection \ref{s2.3}. By Proposition \ref{s2p8}, the metric space $\langle A_0, d\rangle$ is compact. Using (a), one can check that $\{P_0\}$ is a support $\langle A_0, d\rangle$ and, therefore, $\langle A_0, d\rangle\in\mathcal{N}$. Moreover, for every $n\in\mathbb{N}$,  $\{P_0\}$ is a support of $A_{0,n}$. Hence the family $\mathcal{A}=\{A_{0,n+1}: n\in\mathbb{N}\}$ is denumerable in $\mathcal{N}$. We notice that if $M\subseteq\mathbb{N}$, then $\{A_{0,n+1}: n\in M\}\in\mathcal{N}$ because $\{P_0\}$ is a support of $A_{0,n+1}$ for every $n\in M$.  Suppose that $\mathcal{A}$ has a partial choice function in $\mathcal{N}$. Then there exists an infinite set $M\subseteq\mathbb{N}$ such that the family $\mathcal{B}=\{A_{0,n+1}: n\in M\}$ has a choice function in $\mathcal{N}$. Let $f$ be a choice function of $\mathcal{B}$ such that $f\in\mathcal{N}$. Let $D\in [\mathbf{P}]^{<\omega}$ be a support of $f$. Then $D^{\prime}=D\cap\mathbf{P}_0$ is also a support of $f$. Let $n\in\omega$ and $\phi_i\in\mathcal{G}$ with $i\in n+1$ be such that $D^{\prime}=\{\phi_i(P_0): i\in n+1\}$. Since every permutation from $\mathcal{G}$ moves only finitely many atoms, there exists $n_0\in M$ such that $n_0\geq 2$ and $A_{0,n_0}\in\phi_i(P_0)$ for all $i\in n+1$. 

Assume that $f(A_{0,n_0})=x_0$. Since $| A_{0,n_0}|=n_0\geq 2$, there exists $y_0\in A_{0,n_0}$ such that $y_0\neq x_0$. Let $\eta=(x_0, y_0)$, i.e., $\eta$ is the permutation of $A$ which interchanges  $x_0$ and $y_0$, and fixes all other atoms of $\mathcal{N}$. Then $\eta(A_{0,n_0})=A_{n_0}$ and $\eta\in\fix_{\mathcal{G}}(D^{\prime})$. Since $D^{\prime}$ is a support of $f$, we have $\eta(f)=f$. Therefore, since $\langle A_{0,n_0}, x_0\rangle\in f$, we infer that $\langle\eta(A_{0,n_0}), \eta(x_0)\rangle\in \eta(f)=f$, so $\langle A_{0,n_0}, y_0\rangle\in f$ and, in consequence, $x_0=y_0$. The contradiction obtained shows that $\mathcal{A}$ does not have a partial choice function in $\mathcal{N}$. This implies that the set $A_0$ is Dedekind-finite, and, thus, $\mathbf{ICMDI}$ is false in $\mathcal{N}$.

(ii) Let $\mathbf{\Phi}$ be the statement $\mathbf{UT}(\aleph_0, \aleph_0, cuf)\wedge\neg \mathbf{ICMDI}$. We have already noticed that $\neg\mathbf{ICMDI}$ is injectively boundable. This, together  with Proposition \ref{s4p22}, implies  that $\mathbf{\Phi}$ is a conjunction of injectively boundable statements. Therefore, (ii) follows from (i) and from Theorem \ref{pinth}.
\end{proof}

 Clearly, every $\mathbf{ZF}$- model for $\mathbf{UT}(\aleph_0, \aleph_0, cuf)\wedge\neg\mathbf{ICMDI}$ is also a model for $\mathbf{UT}(\aleph_0, \aleph_0, cuf)\wedge\neg\mathbf{M}(C, S)$. This, together with Theorem \ref{s4t23}, implies the following corollary:
 
\begin{corollary} 
\label{s4c24}
The conjunction $\mathbf{UT}(\aleph_0, \aleph_0, cuf)\wedge\neg\mathbf{M}(C, S)$ has a $\mathbf{ZF}$-model.
\end{corollary}

\begin{remark}
\label{s3r3.24}
Let $\mathcal{N}$ be the model that we have used in the proof to Theorem \ref{s4t23}. The proof to Theorem 3.3 in \cite{hdhkr} shows that $\mathbf{UT}(\aleph_0, cuf, cuf)$ is false in $\mathcal{N}$. Hence $\mathbf{CUC}$ is also false in $\mathcal{N}$. Since the statement  $\mathbf{UT}(\aleph_0,\aleph_0, cuf)\wedge\neg\mathbf{UT}(\aleph_0, cuf,cuf)$ has a permutation model (for instance, $\mathcal{N}$), it also has a $\mathbf{ZF}$-model by Proposition \ref{s4p22} and Theorem \ref{pinth}.
\end{remark}

\section{The forms of type $\mathbf{M}(C, \square)$}
\label{s5}

It is known that every separable metrizable space is second-countable in $\mathbf{ZF}$. It is also known, for instance, from Theorem 4.54 of \cite{hh} or from \cite{gt}  that, in $\mathbf{ZF}$, the statement ``every second-countable metrizable space
is separable'' is equivalent  to $\mathbf{CAC}(\mathbb{R})$. The negation of $%
\mathbf{CAC}(\mathbb{R})$ is relatively consistent with $\mathbf{ZF}$, so it is relatively consistent with $\mathbf{ZF}$ that there are non-separable second-countable metrizable spaces. On the other hand, by Theorem \ref{s2t7}(i), it holds in $\mathbf{ZF}$ that separability and second-countability are equivalent in the class of compact metrizable spaces. Theorem \ref{s2t1} shows that totally bounded
metric spaces are second-countable in $\mathbf{ZF}+\mathbf{CAC}$; in particular, it holds in $\mathbf{ZF}+\mathbf{CAC}$ that all compact metrizable spaces are second-countable. 
However, the situation is completely different in $\mathbf{ZF}$. There exist 
$\mathbf{ZF}$-models including compact non-separable metric spaces. Namely, it follows from Theorem \ref{s2t2} that in every $\mathbf{ZF}$-model satisfying the negation of $\mathbf{CAC}_{fin}$, there exists an uncountable, non-separable
compact metric space whose size is incomparable to $|\mathbb{R}|$. The following proposition shows (among other facts) that  $\mathbf{M}(C, \leq|\mathbb{R}|)$  and $\mathbf{M}(C, S)$ are essentially stronger than $\mathbf{CAC}_{fin}$ in $\mathbf{ZF}$ and, furthermore, $\mathbf{IDI}$ is independent of both $\mathbf{ZF}+\mathbf{M}(C, \leq|\mathbb{R}|)$ and $\mathbf{ZF}+\mathbf{M}(C, S)$.  

\begin{proposition}
\label{s5p1}
\begin{enumerate}
\item[(i)] $(\mathbf{ZFA})$ $\mathbf{M}(C, S)\rightarrow \mathbf{M}(C, \leq|\mathbb{R}|)\rightarrow\mathbf{CAC}_{fin}$.
\item[(ii)] $\mathcal{N}_{cr}\models (\neg\mathbf{M}(C, S))\wedge\neg\mathbf{M}(C, \leq|\mathbb{R}|)$. 
\item[(iii)]  $\mathbf{CAC}_{fin}$ implies neither $\mathbf{M}(C, \leq|\mathbb{R}|)$ nor $\mathbf{M}(C, S)$ in $\mathbf{ZF}$.
\item[(iv)]  $\mathbf{IDI}$ implies neither $\mathbf{M}(C, \leq|\mathbb{R}|)$ nor $\mathbf{M}(C, S)$ in $\mathbf{ZF}$.
\item[(v)]  Neither $\mathbf{M}(C, \leq|\mathbb{R}|)$ nor $\mathbf{M}(C, S)$ implies $\mathbf{IDI}$ in $\mathbf{ZF}$.
\end{enumerate}
\end{proposition}
\begin{proof}
(i) It follows directly from Theorem \ref{s2t2} that the implications given in (i) are true in $\mathbf{ZF}$; however, the arguments from \cite{k} are sufficient to show that these implications are also true in $\mathbf{ZFA}$. 

(ii)  By Proposition \ref{s4t8}, there exists a compact metric space $\mathbf{X}=\langle X, d\rangle$ in $\mathcal{N}_{cr}$ such that the set $X$ is not well-orderable in $\mathcal{N}_{cr}$. Since $\mathbb{R}$ is well-orderable in $\mathcal{N}_{cr}$, the set $\mathbf{X}$ is not equipotent to a subset of $\mathbb{R}$ in $\mathcal{N}_{cr}$. Hence $\mathbf{M}(C, \leq|\mathbb{R}|)$ fails in $\mathcal{N}_{cr}$. This, together with (i), implies (ii).

(iii)--(iv) Let $\mathbf{\Phi}$ be either $\mathbf{CAC}_{fin}$ or $\mathbf{IDI}$. In the light of (i), to prove (iii) and (iv), it suffices to show that the conjunction $\mathbf{\Phi}\wedge\neg\mathbf{M}(C, \leq |\mathbb{R}|)$ has a $\mathbf{ZF}$-model. It follows from (ii) that the conjunction $\mathbf{\Phi}\wedge\neg\mathbf{M}(C, \leq|\mathbb{R}|)$ has a permutation model (for instance, $\mathcal{N}_{cr}$). Therefore, since the statements $\mathbf{CAC}_{fin}$, $\mathbf{IDI}$ and $\neg\mathbf{M}(C, \leq|\mathbb{R}|)$ are all injectively boundable,  $\mathbf{\Phi}\wedge\neg\mathbf{M}(C, \leq|\mathbb{R}|)$ has a $\mathbf{ZF}$-model by Theorem \ref{pinth}.  

(v) Let $\mathbf{\Psi}$ be either $\mathbf{M}(C, \leq|\mathbb{R}|)$ or $\mathbf{M}(C, S)$. Since $\mathbf{BPI}$ is true in $\mathcal{N}3$, it follows from Theorem \ref{s2t7} (iii) that $\mathbf{\Psi}$ is true in $\mathcal{N}3$. It is known that $\mathbf{IDI}$ is false in $\mathcal{N}3$. Hence, the conjunction $\mathbf{\Psi}\wedge\neg\mathbf{IDI}$ has a permutation model. To complete the proof, it suffices to apply Theorem \ref{pinth}. 
\end{proof}

\begin{theorem}
\label{s5t02}
$(\mathbf{ZF})$
\begin{enumerate} 
\item[(i)] $(\mathbf{CAC}_{fin}\wedge\mathbf{M}(C,\sigma-l.f))\leftrightarrow\mathbf{M}(C,S)$.
\item[(ii)] $(\mathbf{CAC}_{fin}\wedge \mathbf{M}(C, STB))\leftrightarrow \mathbf{M}(C, S)$.
\end{enumerate}
\end{theorem}
\begin{proof}
Let $\mathbf{X}=\langle X, d\rangle$ be a compact metric space.

($\rightarrow $) We assume both $\mathbf{CAC}_{fin}$ and  $\mathbf{M}(C, \sigma-l.f)$. By our
hypothesis, $\mathbf{X}$ has a base $\mathcal{B=}\bigcup \{\mathcal{B}%
_{n}:n\in \mathbb{N}\}$ such that, for every $n\in \mathbb{N}$, the family $\mathcal{B}%
_{n}$ is locally finite. In $\mathbf{ZF}$, to check that if $\mathcal{A}$ is a locally finite family in $\mathbf{X}$, then it follows from the compactness of $\mathbf{X}$ that $\mathcal{A}$ is finite, we notice that the collection $\mathcal{V}$ of all open sets $V$ of $\mathbf{X}$ such that $V$ meets only finitely many members of $\mathcal{A}$ is an open cover of $\mathbf{X}$, so $\mathcal{V}$ has a finite subcover. In consequence, $X$ meets only finitely many members of $\mathcal{A}$, so $\mathcal{A}$ is finite. Therefore, for every $n\in\mathbb{N}$, the family $\mathcal{B}_n$ is finite. This, together with $\mathbf{CAC}_{fin}$, implies that the family $\mathcal{B}$ is countable, so $\mathbf{X}$ is second-countable.  Hence, by Theorem \ref{s2t7}(i),  $\mathbf{X}$ is
separable as required.\smallskip 

($\leftarrow $) By Proposition \ref{s5p1}, $\mathbf{M}(C,S)$ implies $\mathbf{CAC}_{fin}$. To conclude the proof to (i), it suffices to notice that $\mathbf{M}(C,S)$
implies $\mathbf{M}(C,2)$ and $\mathbf{M}(C,2)$ trivially implies that every
compact metric space has a $\sigma $ - locally finite base.\smallskip 

(ii) ($\rightarrow $) Now, we assume both $\mathbf{CAC}_{fin}$ and $\mathbf{M}(C, STB)$. Since $\mathbf{X}$ is strongly totally
bounded, it follows that it admits a sequence $(D_{n})_{n\in \mathbb{N}}$
such that, for every $n\in \mathbb{N}$, $D_{n}$ is a $\frac{1}{n}$-net of $\mathbf{X}$.
By $\mathbf{CAC}_{fin}$,  the set $D=\bigcup \{D_{n}:n\in \mathbb{N}\}$ is
countable. Since, $D$ is dense in $\mathbf{X}$, it follows that $\mathbf{X}$ is
separable.

($\leftarrow $) It is straightforward to check that every separable compact metric space is strongly totally bounded. Hence $\mathbf{M}(C,S)$ implies $\mathbf{M}(C, STB)$. Proposition \ref{s5p1} completes the proof.
\end{proof}

\begin{remark}
\label{s5r03}
By Proposition \ref{s5p1}, there exists a model $\mathcal{M}$ of $\mathbf{ZF}$ in which $\mathbf{CAC}_{fin}$ holds and $\mathbf{M}(C,S)$ fails. By Theorem \ref{s5t02}(i), $\mathbf{M}(C, \sigma-l.f)$ fails in $\mathcal{M}$. This, together with Theorem \ref{s2t5}(i), implies that $\mathbf{M}(C, \sigma-l.f.)$ independent of $\mathbf{ZF}$. 
\end{remark}

\begin{theorem}
\label{s5t2} $(\mathbf{ZF})$
\begin{enumerate}
 \item[(i)] $(\mathbf{CAC}(\mathbb{R},C)\wedge\mathbf{M}(C, \leq |\mathbb{R}|))\leftrightarrow %
\mathbf{M}(C, S).$
\item[(ii)] $\mathbf{CAC}(\mathbb{R})$ does not imply $\mathbf{M}(C, \leq
|\mathbb{R}|)$.
\end{enumerate}
\end{theorem}
\begin{proof}
(i) ($\rightarrow $)  We assume $\mathbf{CAC}(\mathbb{R},C)$ and $\mathbf{M}(C, \leq |\mathbb{R}|)$. We fix a  compact metric space $\mathbf{X}=\langle X,d\rangle$ and prove that $\mathbf{X}$ is separable. For every $n\in \mathbb{N}$, let $\mathbf{X}^{n}=\langle X^n, d_n\rangle$ where $d_n$ is the metric on $X^n$ defined by:

$$d_{n}(x,y)=\max \{d(x(i),y(i)):i\in n\}.$$

\noindent Then $\mathbf{X}^n$ is compact for every $n\in\mathbb{N}$.  By $\mathbf{M}(C, \leq |\mathbb{R}%
|) $, $|X|\leq |\mathbb{R}|$. Therefore, since  $|X^{\mathbb{N}}|\leq |\mathbb{R}|$, there exists a family $\{\psi_n: n\in\mathbb{N}\}$ such that, for every $n\in\mathbb{N}$, $\psi_n: X^n\to\mathbb{R}$ is an injection. The metric $d$ is totally bounded, so, for every $n\in 
\mathbb{N} $, the set $$M_n=\{m\in \mathbb{N}:\exists y\in X^{m},\forall x\in
X,d(x,\{ y(i): i\in n\})<\frac{1}{n}\}$$
\noindent is non-empty. Let $k_n=\min M_n$ for every $n\in\mathbb{N}$. To prove that $\mathbf{X}$ is strongly totally bounded, for every $n\in\mathbb{N}$, we consider the set $C_n$ defined as follows:
\[
C_{n}=\{y\in X^{k_{n}}:\forall x\in X (d(x,\{y(i): i\in k_n\})<\frac{1}{n})\}.
\]%
We claim that for every $n\in \mathbb{N},C_{n}$ is a closed subset of $%
\mathbf{X}^{k_{n}}$. To this end, we fix $y_0\in X^{k_{n}}\setminus C_{n}$.
Then, since $X$ is infinite, there exists $x_{0}\in X$ such that $B_{d}(x_{0},\frac{1}{n})\cap \{y_0(i): i\in n\}=\emptyset $. Let $%
r=d(x_{0},\{y_0(i): i\in n\})$ and $\varepsilon =r-\frac{1}{n}$. Then $\varepsilon >0$. To show that $B_{d_{k_n}}(y_0,%
\varepsilon )\cap C_{n}=\emptyset $, suppose that $z_0\in B_{d_{k_n}}(y_0,\varepsilon )\cap
C_{n}$. Then 
\[
d(x_{0},\{z_0(i): i\in k_n\})=\max \{d(x_{0}, z_0(i)): i\in k_{n}\}<\frac{1}{n} 
\]%
and 
\[
d_{k_n}(y_0, z_0)=\max \{d( y_0(i), z_0(i)):i\in k_{n}\}<\varepsilon.
\]%
For every $i\in k_n$, we have: 
\[
r\leq d(x_{0},y_0(i))\leq d(x_{0}, z_0(i))+d(z_0(i),y_0(i))\leq d(x_0, z_0(i))+d_{k_n}(z_0, y_0).
\]%
Hence, for every $i\in k_n$, the following inequalities hold: 
\[
r-d_{k_n}(z_0,y_0)\leq d(x_{0}, z_0(i))<\frac{1}{n}.
\]%
\noindent In consequence, $\varepsilon<d_{k_n}(z_{0},y_0)$. The contradiction obtained shows that $B_{d_{k_n}}(y_0,%
\varepsilon )\cap C_{n}=\emptyset$. Hence, for every $n\in\mathbb{N}$, the non-empty set $C_n$ is compact in the metric space $\mathbf{X}^{k_n}$. Therefore, it follows from $\mathbf{CAC}(\mathbb{R}, C)$ that the family $\{\psi_n(C_n): n\in\mathbb{N}\}$ has a choice function. This implies that $\{C_n: n\in\mathbb{N}\}$ has a choice function, so we can fix $f\in\prod\limits_{n\in\mathbb{N}}C_n$. Then, for every $n\in\mathbb{N}$, the set $D_n=\{f(n)(i): i\in k_n\}$ is a $\frac{1}{n}$-net in $\mathbf{X}$. This shows that $\mathbf{X}$ is strongly totally bounded. It is easily seen that the set $D=\bigcup\limits_{n\in\mathbb{N}}D_n$ is countable and dense in $\mathbf{X}$. Hence 
$\mathbf{CAC}(\mathbb{R}, C)\wedge\mathbf{M}(C, \leq |\mathbb{R}|)$ implies $\mathbf{M}(C, S)$.\smallskip

($\leftarrow $) By Proposition \ref{s5p1}(i), $\mathbf{M}(C, S)$ implies $\mathbf{M}(C, \leq|%
\mathbb{R}|)$. Assuming $\mathbf{M}(C, S)$, we prove that $\mathbf{CAC}(\mathbb{R},C)$ holds. To this aim, we fix a disjoint family $%
\mathcal{A}=\{A_{n}:n\in \mathbb{N}\}$ of non-empty subsets of $\mathbb{R}$ such that there exists a family $\{\rho_{n}:n\in \mathbb{N}\}$ of metrics such that, for every $n\in \mathbb{N}$, $\langle A_{n},\rho_{n}\rangle$ is a compact metric space. Let $A=\bigcup\limits_{n\in\mathbb{N}}A_n$, let $\infty\notin A$ and $X=A\cup\{\infty\}$. Let $d$ be the metric on $X$ defined by ($\ast$) in Subsection \ref{s2.3}. Then, by Proposition \ref{s2p8},  $\mathbf{X}=\langle X, d\rangle$ is a compact metric space. It follows from $\mathbf{M}(C, S)$ that $\mathbf{X}$ is separable. Let $H=\{x_n: n\in\mathbb{N}\}$ be a dense set in $\mathbf{X}$. For every $n\in\mathbb{N}$, let $m_n=\min\{m\in\mathbb{N}: x_m\in A_n\}$ and let $h(n)=x_{m_n}$. Then $h$ is a choice function of $\mathcal{A}$. Hence $\mathbf{M}(C, S)$ implies $\mathbf{CAC}(\mathbb{R},C)$. \medskip

(ii) It was shown in \cite{k} that $\mathbf{CAC}(\mathbb{R})$ and $\mathbf{M}(C, S)$ are independent of each other.  Since $\mathbf{M}(C, S)$ implies $\mathbf{M}(C, \leq|\mathbb{R}|)$ (see Proposition \ref{s5p1}(i)), while  $\mathbf{CAC}(%
\mathbb{R})$ implies $\mathbf{CAC}(\mathbb{R},C)$ but not $\mathbf{M}(C, S)$ the conclusion follows from (i).
\end{proof}

\begin{corollary}
\label{s5c3}
In every permutation model, $\mathbf{M}(C, \leq|\mathbb{R}|)$ and $\mathbf{M}(C, S)$ are equivalent.
\end{corollary}
\begin{proof}
Let $\mathcal{N}$ be a permutation model. Since $\mathbb{R}$ is well-orderable in $\mathcal{N}$ (see Subsection \ref{s2.4}),  $\mathbf{CAC}(\mathbb{R})$ is true in $\mathcal{N}$. This, together with Theorem \ref{s5t2}(i), completes the proof.
\end{proof}

\begin{remark}
\label{s5r4}
(i) The proof to Corollary \ref{s5c3} shows that $\mathbf{M}(C, \leq|\mathbb{R}|)$ and $\mathbf{M}(C, S)$  are equivalent in every model of $\mathbf{ZFA}$ in which $\mathbb{R}$ is well-orderable.

(ii) In much the same way, as in the proof to Theorem \ref{s5t2}(ii)($\leftarrow$), one can show that, for every family $\{\mathbf{X}_{n}: n\in\omega\}$ of pairwise disjoint compact spaces, if the direct sum $\mathbf{X}=\bigoplus\limits_{n\in\omega}\mathbf{X}_n$ is metrizable, then it is separable.
\end{remark}

We include a sketch of a $\mathbf{ZF}$-proof to the following lemma for completeness. We use this lemma in our $\mathbf{ZF}$-proof that $\mathbf{M}(C, \hookrightarrow \lbrack 0,1]^{\mathbb{N}})$ and $\mathbf{M}(C, S)$ are equivalent.

\begin{lemma}
\label{s5l5}
 $(\mathbf{ZF})$
Suppose that $\mathcal{B}$ is a base of a non-empty metrizable space $\mathbf{X}=\langle X, \tau\rangle$. Then there exists a homeomorphic embedding of $\mathbf{X}$ into the cube $[0, 1]^{\mathcal{B}\times\mathcal{B}}$.
\end{lemma}
\begin{proof} We may assume that $X$ consists of at least two points. Let $d$ be a metric on $X$ such that $\tau=\tau(d)$ and let 
\[
\mathcal{W}=\{\langle U,V\rangle\in \mathcal{B}\times \mathcal{B}:\emptyset\neq\text{cl}_{\mathbf{X}}{U}\subseteq
V\neq X\}.
\]%
For every  $W=\langle U,V\rangle\in \mathcal{W%
}$, by defining 
\[
f_W(x)=\frac{d(x,\text{cl}_{\mathbf{X}}({U}))}{d(x,\text{cl}_{\mathbf{X}}({U}))+d(x,X\setminus V)} \text{ whenever } x\in X, 
\]%
we obtain a continous function from $\mathbf{X}$ into $[0, 1]$. Let $h:X\to [0,1]^{\mathcal{W}}$ be the evaluation mapping defined by: $h(x)(W)=f_W(x)$ for all $x\in X$ and $W\in\mathcal{W}$. Then $h$ is a homeomorphic embedding of $\mathbf{X}$ into $[0, 1]^{\mathcal{W}}$.  To complete the proof, it suffices to notice that $[0, 1]^{\mathcal{W}}$ is homeomorphic to a subspace of $[0,1]^{\mathcal{B}\times\mathcal{B}}$.
\end{proof}

\begin{theorem}
\label{s5t6}
$(\mathbf{ZF})$
\begin{enumerate}
\item[(i)] $\mathbf{M}(C, \hookrightarrow \lbrack 0,1]^{\mathbb{N}})\leftrightarrow \mathbf{M}(C,S)$.
\item[(ii)] $\mathbf{M}(C, \leq |\mathbb{R}|)\rightarrow\mathbf{M}(C, \hookrightarrow
\lbrack 0,1]^{\mathbb{R}})$.
\end{enumerate}
\end{theorem}

\begin{proof}
Let $\mathbf{X}=\langle X, \tau\rangle$ be an infinite compact metrizable space and let $d$ be a metric on $X$ such that $\tau=\tau(d)$.

(i) \textbf{(}$\rightarrow $) We assume $\mathbf{M}(C, \hookrightarrow \lbrack 0,1]^{\mathbb{N}})$ and show that $\mathbf{X}$ is separable.  
 By our hypothesis, $\mathbf{X}$ is homeomorphic to a compact subspace $\mathbf{Y}$ of the Hilbert cube $[0,1]^{\mathbb{N}}$. Since $[0,1]^{\mathbb{N}}$ is second-countable, it follows from Theorem \ref{s2t7}(i) that $\mathbf{Y}$ is separable. Hence $\mathbf{X}$ is separable. In consequence, $\mathbf{M}(C, \hookrightarrow \lbrack 0,1]^{\mathbb{N}})$ implies $\mathbf{M}(C, S)$.

($\leftarrow $) If $\mathbf{M}(C, S)$ holds, then every compact metrizable space is second-countable. Since, by Lemma \ref{s5l5}, every second-countable metrizable space is embeddable in the Hilbert cube $[0,1]^{\mathbb{N}}$, $\mathbf{M}(C, S)$ implies $\mathbf{M}(C,\hookrightarrow \lbrack 0,1]^{\mathbb{N}})$.

(ii) Now, suppose that $\mathbf{M}(C, \leq |\mathbb{R}|)$ holds. Then  $|X|\leq |\mathbb{R}%
|$. Since $|[\mathbb{R}]^{<\omega }|=|\mathbb{R}|$, we infer that $%
|[X]^{<\omega }|\leq |\mathbb{R}|$. For every $n\in \mathbb{N}$, let 
\[
k_{n}=\min \{m\in \mathbb{N}:\bigcup\limits_{x\in A}B_d(x,\frac{1}{n})=X \text{ for some 
}A\in \lbrack X]^{m}\}
\]%
\noindent and 
\[
E_{n}=\{A\in \lbrack X]^{k_{m}}:\bigcup\limits_{x\in A}B_d(x,\frac{1}{n})=X\}.
\]%
Let $\mathcal{B}=\{B_d(x,\frac{1}{n}):x\in A,A\in E_{n},n\in \mathbb{N}\}$. It is
straightforward to verify that $\mathcal{B}$ is a base for $\mathbf{X}$ of
size $|\mathcal{B}|\leq |\mathbb{R}\times \mathbb{N}|\leq |\mathbb{R}|$, so $|\mathcal{B}\times\mathcal{B}|\leq|\mathbb{R}|$. This, together with Lemma \ref{s5l5}, implies that $\mathbf{X}$ is embeddable into $[0,1]^{\mathbb{R}}$. Hence $\mathbf{M}(C, \leq |\mathbb{R}|)$ implies $\mathbf{M}(C, \hookrightarrow
\lbrack 0,1]^{\mathbb{R}})$. 
\end{proof}

In view of Theorem \ref{s5t6}, one may ask the following
questions:

\begin{question}
\label{s5q7} 
\begin{enumerate}
\item[(i)] Does $\mathbf{M}(C, \hookrightarrow \lbrack 0,1]^{\mathbb{R}})$ imply $%
\mathbf{M}(C, \hookrightarrow \lbrack 0,1]^{\mathbb{N}})$?

\item[(ii)] Does $\mathbf{M}(C, \hookrightarrow \lbrack 0,1]^{\mathbb{R}})$ imply $%
\mathbf{M}(C, \leq |\mathbb{R}|)$?

\item[(iii)] Does $\mathbf{CAC}_{fin}$ imply $\mathbf{M}(C, \hookrightarrow \lbrack
0,1]^{\mathbb{R}})$?

\item[(iv)] Does $\mathbf{M}(C, \hookrightarrow \lbrack
0,1]^{\mathbb{R}})$ imply $\mathbf{CAC}_{fin}$?
\end{enumerate}
\end{question}

\begin{remark}
\label{s5r8}  ($a$) Regarding Question \ref{s5q7} (i)-(ii), we notice that the answer is in the affirmative in permutation models. Indeed, let $\mathcal{N}$ be a permutation model. It is known  that $\mathbb{R}$ and $\mathcal{P}(\mathbb{R})$ are well-orderable in $\mathcal{N}$ (see Subsection \ref{s2.4}).
Therefore, assuming that  $\mathbf{%
M}(C,\hookrightarrow \lbrack 0,1]^{\mathbb{R}})$ holds in $\mathcal{N}$ and working inside $\mathcal{N}$,  we deduce that, given a compact metrizable space $\mathbf{X}$ in $\mathcal{N}$, $\mathbf{X}$ embeds in $%
[0,1]^{\mathbb{R}}$. Hence $\mathbf{X}$ is a well-orderable space, so $\mathbf{X}$ is Loeb. Since $\mathbf{X}$ is a compact metrizable Loeb space, by Theorem \ref{s2t7}(i), $\mathbf{X}$ is second-countable.  Therefore, by Lemma \ref{s5l5}, $%
\mathbf{X}$ embeds in $[0,1]^{\mathbb{N}}$ and, consequently, $|X|\leq |%
\mathbb{R}|$.

($b$) Regarding Question \ref{s5q7}(iii), we note that $\mathbf{CAC}_{fin}$ holds in the permutation model $\mathcal{N}_{cr}$. To show that $\mathbf{M}(C, \hookrightarrow \lbrack 0,1]^{%
\mathbb{R}})$ fails in $\mathcal{N}_{cr}$, we  notice that, by Proposition \ref{s4t8}, there exists a compact metric space $\mathbf{X}=\langle X, d\rangle$ in $\mathcal{N}_{cr}$ such that $X$ is not well-orderable in $\mathcal{N}_{cr}$. Since $[0,1]^{\mathbb{R}}$, being
equipotent to the well-orderable set $\mathcal{P}(\mathbb{R})$ of $%
\mathcal{N}_{cr}$, is well-orderable in $\mathcal{N}_{cr}$, it is true in $\mathcal{N}_{cr}$ that $\mathbf{X}$ is not embeddable in $[0,1]^{%
\mathbb{R}}$. This explains why $\mathbf{M}(C, \hookrightarrow \lbrack 0,1]^{\mathbb{R}})$
fails in $\mathcal{N}_{cr}$. Therefore, since the conjunction $\mathbf{CAC}_{fin}\wedge\neg \mathbf{M}(C, \hookrightarrow [0,1]^{\mathbb{R}})$ has a permutation model, it also has a $\mathbf{ZF}$-model by Theorem \ref{pinth}.
\end{remark}

To give more light to Questions \ref{s5q7} (iii)-(iv), let us prove the following Theorems \ref{s5t9} and \ref{s5t10}.

\begin{theorem}
\label{s5t9}
$(\mathbf{ZF})$
\begin{enumerate}
\item[(i)] $\mathbf{M}(C, |\mathcal{B}_Y|\leq |\mathcal{B}|)\rightarrow\mathbf{M}([0,1],|\mathcal{B}_Y|\leq |\mathcal{B}|)\rightarrow \mathbf{IDI}(\mathbb{R})$.
\item[(ii)] The following are equivalent: 
\begin{enumerate}
\item[(a)] Every compact (0-dimesional) subspace of the Tychonoff cube $[0,1]^{%
\mathbb{R}}$ has a base of size $\leq |\mathbb{R}|$;
\item[(b)] every compact (0-dimensional) subspace of the Tychonoff cube $[0,1]^{%
\mathbb{R}}$ with a unique accumulation point has a base of size $\leq |\mathbb{R}|$%
;
\item[(c)] $\mathbf{Part}(\mathbb{R})$.
\end{enumerate}
\item[(iii)] Every compact metrizable subspace of the Tychonoff cube $[0,1]^{\mathbb{R}%
}$ with a unique accumulation point has a base of size $\leq |\mathbb{R}|$ iff for
every denumerable family $\mathcal{A}$ of finite subsets of $%
\mathcal{P}(\mathbb{R})$ such that $\bigcup \mathcal{A}$ is pairwise
disjoint, $|\bigcup \mathcal{A}|\leq $ $|\mathbb{R}|$.
\end{enumerate}
\end{theorem}
\begin{proof}
(i) Assume that $\mathbf{IDI}(\mathbb{R})$ is false.  By a well-known result of N. Brunner (cf. \cite[Form \text{[13 A]}]{hr}), there exists a Dedekind-finite dense subset of the interval $(0, 1)$ with its usual topology. Then 
\[
\mathcal{B}=\{(x,y):x,y\in D,x<y\}\cup\{[0, x): x\in D\}\cup\{ (x, 1]: x\in D\}
\]%
is a base for the usual topology of $[0, 1]$. Clearly, the set $\mathcal{B}$ is Dedekind-finite. Let us consider the compact subspace  $\mathbf{Y}$ of $[0, 1]$ where 
\begin{equation}
Y=\{0\}\cup \{\frac{1}{n}:n\in \mathbb{N}\}.  \label{13}
\end{equation}%
Since $\{\{\frac{1}{n}\}: n\in\mathbb{N}\}\subseteq \mathcal{B}_Y$, the set $\mathcal{B}_Y$ is Dedekind-infinite. Therefore, if  $|\mathcal{B}%
_{Y}|\leq |\mathcal{B}|$, then $\mathcal{B}$ is Dedekind-infinite but this is impossible. Hence $\mathbf{M}([0,1], |\mathcal{B}_Y|\leq|\mathcal{B}|)$ implies $\mathbf{IDI}(\mathbb{R})$. It is clear that $\mathbf{M}(C, |\mathcal{B}_Y|\leq |\mathcal{B}|)$ implies $\mathbf{M}([0,1], |\mathcal{B}_Y|\leq |\mathcal{B}|)$. This completes the proof to (i).

(ii) It is obvious that $(a)$ implies $(b)$. \smallskip 

$(b) \rightarrow  (c)$  Fix a partition $\mathcal{P}$ of $%
\mathbb{R}$. That is, $\mathcal{P}$ is a disjoint family of non-empty subsets of $\mathbb{R}$ such that $\mathbb{R}=\bigcup\mathcal{P}$. Assuming $(b)$,  we show that $|\mathcal{P}|\leq |\mathbb{R}|$. For $P\in\mathcal{P}$, let $f_{P}:\mathbb{R}\to\{0, 1\}$ be the
characteristic function of $P$ and let $f(x)=0$ for each $x\in\mathbb{R}$. We put
\[
X=\{f\}\cup \{f_{P}:P\in \mathcal{P}\}.
\]%
We claim that the subspace $\mathbf{X}$ of $[0,1]^{\mathbb{R}}$ is compact.
To see this, let us consider an arbitrary family $\mathcal{U}$ of open subsets of $[0,1]^{\mathbb{R}}$ such that $X\subseteq\bigcup\mathcal{U}$. There exists $U_0\in\mathcal{U}$ such that $f\in U_0$. There exist $\varepsilon\in(0,1)$ and  a non-empty finite subset $J$ of $\mathbb{R}$ such that the set 
\[
V=\bigcap \{\pi _{j}^{-1}([0,\varepsilon )):j\in J\}
\]%
is a subset of $U_0$ where, for each $j\in \mathbb{R}$ and $x\in [0,1]^{\mathbb{R}}$, $\pi_j(x)=x(j)$. Since $J$ is finite, there exists a finite set $\mathcal{P}_{J}\subseteq\mathcal{P}$ such that $J\subseteq\bigcup\mathcal{P}_J$. We notice that,
for every $P\in \mathcal{P}\setminus\mathcal{P}_J$ and every $j\in J$, $f_{P}(j)=0$. Hence $f_P\in U_0$ for every $P\in\mathcal{P}\setminus\mathcal{P}_J$. This implies that there exists a finite set $\mathcal{W}$ such that $\mathcal{W}\subseteq\mathcal{U}$ and $X\subseteq\bigcup\mathcal{W}$. Hence $\mathbf{X}$ is compact as claimed. If $P\in \mathcal{P}$ and $j\in P$, then $\pi _{j}^{-1}((\frac{1}{2},1])\cap X=\{f_{P}\}$, so $f_{P}$ is an isolated point of $\mathbf{X}$. Hence $f$ is the unique accumulation point of $\mathbf{X}$. The space $\mathbf{X}$ is also 0-dimensional. By $(b)$, $\mathbf{X}$ has a base  $\mathcal{B}$ equipotent to a subset of $\mathbb{R}$. Since $\{\{f_P\}: P\in\mathcal{P}\}\subseteq\mathcal{B}$, it follows that $|\mathcal{P}|\leq|\mathbb{R}|$ as required.\smallskip\ 

$(c) \rightarrow (a)$ We assume $\mathbf{Part}(\mathbb{R})$ and  fix a compact subspace $\mathbf{X}$ of the cube 
$[0,1]^{\mathbb{R}}$. It is well  known that $[0,1]^{\mathbb{R}}$ is
separable in $\mathbf{ZF}$ (cf., e.g., \cite{kkk}).  Fix a countable dense subset $D$ of $[0,1]^{%
\mathbb{R}}$. For every $y\in D$, let 
\[
\mathcal{V}_{y}=\{\bigcap \{\pi _{i}^{-1}((y(i)-1/m,y(i)+1/m)): i\in F\}: \emptyset\neq F\in
\lbrack \mathbb{R}]^{<\omega },m\in \mathbb{N}\}.
\]%
Since $|[\mathbb{R}]^{<\omega }|=|\mathbb{R}\times \mathbb{N}|=|\mathbb{R}|$
in $\mathbf{ZF}$, it follows that $\mathcal{B}=\bigcup \{\mathcal{V}%
_{y}:y\in D\}$ is equipotent to $\mathbb{R}$. It is a routine work to verify that $%
\mathcal{B}$ is a base for $[0,1]^{\mathbb{R}}$. Define an equivalence
relation $\sim $ on $\mathcal{B}$ by requiring: 
\begin{equation}
O\sim Q\text{ iff }O\cap X=Q\cap X\text{.}  \label{11}
\end{equation}%
Clearly 
\begin{equation}
\mathcal{B}_{X}=\{P\cap X:[P]\in \mathcal{B}/\sim \}  \label{12}
\end{equation}%
is a base for $\mathbf{X}$ of size $|\mathcal{B}/\sim |$. Since $|\mathcal{B%
}/\sim |\leq |\mathbb{R}|$, it follows that $|\mathcal{B}_{X}|\leq |\mathbb{R}%
|$ as required.\smallskip 

(iii) ($\rightarrow $) Fix family $\mathcal{A}=\{A_{n}:n\in \mathbb{N}\}$ of
finite subsets of $\mathcal{P}(\mathbb{R})$ such that the family $\mathcal{P}_0=\bigcup \mathcal{A}$ is pairwise disjoint. Let $\mathcal{P}_1=\mathcal{P}_0\cup\{\mathbb{R}\setminus\bigcup\mathcal{P}_0\}$ and $\mathcal{P}=\mathcal{P}_1\setminus\{\emptyset\}$. Then $\mathcal{P}$ is a partition of $\mathbb{R}$. Let $f, f_P$ with $P\in\mathcal{P}$ and $X$ be defined as in the proof of (ii) that $(b)$ implies $(c)$. Since $\mathcal{P}$ is a cuf set, the space $\mathbf{X}$ has a $\sigma$-locally finite base. This, together with Theorem \ref{s2t5}(ii), implies that $\mathbf{X}$ is metrizable. Suppose $\mathbf{X}$ has a base $\mathcal{B}$ such that $|\mathcal{B}|\leq |\mathbb{R}|$. In much the same way, as in the proof that $(b)$ implies $(c)$ in (ii), we can show that $|\mathcal{P}|\leq |\mathbb{R}|$. Then $|\bigcup\mathcal{A}|\leq |\mathbb{R}|$.  \smallskip 

($\leftarrow $) Now, we consider an arbitrary compact metrizable subspace $\mathbf{X}$ of the cube $[0,1]^{\mathbb{R}}$ such that $\mathbf{X}$ has a unique accumulation point. Let $x_0$ be the accumulation point of $\mathbf{X}$ and let $d$ be a metric on $X$ which induces the topology of $\mathbf{X}$.  For every $x\in X\backslash \{x_0\}$ let 
\[
n_{x}=\min \{n\in \mathbb{N}: B_d(x,\frac{1}{n})=\{x\}\}.
\]%
For every $n\in \mathbb{N}$, let 
\[
E_{n}=\{x\in X:n_{x}=n\}\text{.}
\]%
Without loss of generality, we may assume that, for every $n\in \mathbb{N}$, $E_{n}\neq \emptyset $. Since $\mathbf{X}$ is compact, it follows easily
that, for every $n\in\mathbb{N}$, the set $E_{n}$ is finite. Let us apply the base $\mathcal{B}$ 
of $[0,1]^{\mathbb{R}}$ given in the proof of part (ii) that $(c)$ implies $(a)$. Let $\sim $ be the equivalence relation on $\mathcal{B}$ given by (\ref{11}). Let 
\[
\mathcal{C}_{X}=\{[P]\in \mathcal{B}/\sim \text{ }:|P\cap (X\setminus
\{x_0\})|=1\}.
\]%
For every $n\in \mathbb{N}$, let 
\[
C_{n}=\{[P]\in \mathcal{C}_{X}:P\cap E_{n}\neq \emptyset \}.
\]%
Clearly, for every $n\in \mathbb{N}$, $|C_{n}|=|E_{n}|$. Let $\mathcal{C}=\bigcup\{C_n: n\in\mathbb{N}\}$. There exists a bijection $\psi:\mathcal{B}\to\mathbb{R}$. For every $n\in\mathbb{N}$. we put $A_n=\{\{\psi(U): U\in H\}: H\in C_n\}$. Then, for every $n\in\mathbb{N}$, $A_n$ is a finite subset of $\mathcal{P}(\mathbb{R})$. Let $\mathcal{A}=\{A_n: n\in\mathbb{N}\}$. Then $\bigcup\mathcal{A}$ is pairwise disjoint. Suppose that $|\bigcup\mathcal{A}|\leq|\mathbb{R}|$. Then  $|\mathcal{C}|\leq $ $|\mathbb{R}|$. This implies that the family $\mathcal{W}=\{P\cap ( X\setminus\{x_0\}): [P]\in \mathcal{C}\}$ is of  size $\leq $ $|\mathbb{R}|$. The family $\mathcal{G}=\mathcal{W}\cup\{B_d(x_0, \frac{1}{n}): n\in\mathbb{N}\}$ is a base of $\mathbf{X}$ such that $|\mathcal{G}|\leq |\mathbb{R}|$. 
\end{proof}

The following theorem leads to a partial answer to Question \ref{s5q7}(iv).

\begin{theorem}
\label{s5t10}
$(\mathbf{ZF})$
\begin{enumerate}
\item[(i)] $(\mathbf{M}(C, \hookrightarrow \lbrack 0,1]^{\mathbb{R}})\wedge \mathbf{Part}(\mathbb{R}))\leftrightarrow \mathbf{M}(C, B(\mathbb{R}))$.
\item[(ii)] $\mathbf{M}(C, S)\rightarrow \mathbf{M}(C, W(\mathbb{R}))\rightarrow%
\mathbf{M}(C, B(\mathbb{R}))\rightarrow \mathbf{CAC}_{fin}$.
\item[(iii)] $(\mathbf{CAC}(\mathbb{R})\wedge \mathbf{M}(C, B(\mathbb{%
R}))\rightarrow \mathbf{M}(C, S)$. 
\item[(iv)]  $\mathbf{CAC}(\mathbb{R})\rightarrow (\mathbf{M}(C, S)\leftrightarrow \mathbf{M}(C, W(\mathbb{R}))\leftrightarrow \mathbf{M}(C, B(%
\mathbb{R}))$.
\end{enumerate}
\end{theorem}

\begin{proof}
(i) This follows from Theorem \ref{s5t9}(ii) and Lemma \ref{s5l5}.\smallskip 

(ii) It is obvious that $\mathbf{M}(C, W(\mathbb{R}))$ implies $\mathbf{M}(C, B(\mathbb{R}))$. Assume  $\mathbf{M}(C, S)$ and let $\mathbf{Y}=\langle Y, \tau\rangle$ be an infinite compact metrizable separable space. Since $\mathbf{Y}$ is second-countable and $|\mathbb{R}^{\omega}|=|\mathbb{R}|$, it follows that $|\tau|\leq|\mathbb{R}|$. To show that $|\mathbb{R}|\leq |\tau|$, we notice that, since $X$ is infinite and $\mathbf{X}$ is second-countable, there exists a disjoint family $\{U_n: n\in\omega\}$ such that, for each $n\in\omega$, $U_n\in\tau$. For $J\in\mathcal{P}(\omega)$, we put $\psi(J)=\bigcup\{U_n: n\in J\}$ to obtain an injection $\psi:\mathcal{\omega}\to\tau$. Hence $|\mathbb{R}|=|\mathcal{P}(\omega)|\leq|\tau|$. 

 To see that $\mathbf{M}(C, B(\mathbb{R}))\rightarrow 
\mathbf{CAC}_{fin}$, we assume $\mathbf{M}(C, B(\mathbb{R}))$, fix a disjoint family $\mathcal{A}%
=\{A_{n}:n\in \mathbb{N}\}$ of non-empty finite sets and show that $\mathcal{A}$ has a choice function. To this aim, we put $A=\bigcup\mathcal{A}$, take an element $\infty\notin A$ and $X=A\cup \{\infty \}$. For each $n\in\mathbb{N}$, let $\rho_n$ be the discrete metric on $A_n$. Let $d$ be the  metric
on $X$ defined by ($\ast$) in Subsection \ref{s2.3}. By our hypothesis, the space $\mathbf{X}=\langle X, d\rangle$
has a base $\mathcal{B}$ of size $\leq |\mathbb{R}|$. Let $\psi:\mathcal{B}\to\mathbb{R}$ be an injection.  Since $\{\{x\}:x\in
A\}\subseteq \mathcal{B}$ and the sets $A_n$ are finite, for each $n\in\mathbb{N}$,  we can define $A_n^{\star}=\{\psi(\{x\}): x\in A_n\}$ and $a_n^{\star}=\min A_n^{\star}$. For each $n\in\mathbb{N}$, there is a unique $x_n\in A_n$ such that $\psi(\{x_n\})=a_n^{\star}$. This shows that $\mathcal{A}$ has a choice function. \smallskip 

(iii) Now, we assume both $\mathbf{CAC}(\mathbb{R})$ and $\mathbf{M}(C, B(\mathbb{R}))$. Let us consider an arbitrary  compact metric space $\mathbf{X}=\langle X, \rho\rangle$. By our hypothesis, $%
\mathbf{X}$ has a base $\mathcal{B}$ \ of size $\leq |\mathbb{R}|$. Since, $%
|[\mathbb{R}]^{<\omega }|\leq |\mathbb{R}|$, it follows that $|[\mathcal{B}%
]^{<\omega }|\leq |\mathbb{R}|$. For every $n\in \mathbb{N},$ let 
\[
\mathcal{A}_{n}=\{\mathcal{F}\in \lbrack \mathcal{B}]^{<\omega }:\bigcup \mathcal{F}=X\wedge\forall F\in \mathcal{F}  (\delta_{\rho}(F)\leq\frac{1}{n})\}.
\]%
Since $\mathbf{X}$ is compact, $\rho$ is totally bounded. Therefore, $%
\mathcal{A}_{n}\neq \emptyset $ for every $n\in \mathbb{N}$. By $\mathbf{CAC}(%
\mathbb{R})$, we can fix a sequence $(\mathcal{F}_n)_{n\in\mathbb{N}}$ such that, for every $n\in \mathbb{N}$, $\mathcal{F}_n\in\mathcal{A}_n$. Since $|[\mathcal{B}]^{<\omega }|\leq |\mathbb{R}|$, we can also fix a sequence $(\leq_n)_{n\in\mathbb{N}}$ such that, for every $n\in\mathbb{N}$, $\leq_n$ is a well-ordering on $\mathcal{F}_n$. This implies that the family 
$\mathcal{F}_0=\bigcup \{\mathcal{F}_{n}:n\in \mathbb{N}\}$
is countable. Furthermore, it
is a routine work to verify that $\mathcal{F}_0$ is a base of $\mathbf{X}$. Hence $\mathbf{X}$ is second-countable. By Theorem \ref{s2t7} (i), $\mathbf{X}$ is separable. This completes the proof to (iii).\smallskip

That (iv) holds follows directly from (ii) and (iii). 
\end{proof}

Our proof to the following theorem emphasizes the usefulness of Theorems \ref{s3t1} and \ref{s5t10}:

\begin{theorem}
\label{s5t11}
\begin{enumerate}
\item[($a$)] The following implications are true in $\mathbf{ZF}$:
$$\mathbf{WO}(\mathcal{P}(\mathbb{R}))\rightarrow\mathbf{WO}(\mathbb{R})\rightarrow \mathbf{{Part}}(\mathbb{R})\rightarrow (\mathbf{M}(C, \hookrightarrow [0, 1]^{\mathbb{R}})\rightarrow\mathbf{CAC}_{fin}).$$
\item[($b$)] There exists a symmetric model of $\mathbf{ZF}+\mathbf{CH}+\mathbf{WO}(\mathcal{P}(\mathbb{R}))$ in which  $\mathbf{M}(C, \hookrightarrow [0, 1]^{\mathbb{R}})$ is false. Hence, $\mathbf{M}(C, \hookrightarrow [0, 1]^{\mathbb{R}})$ does not follow from $\mathbf{Part}(\mathbb{R})$ in $\mathbf{ZF}$.
\end{enumerate}
\end{theorem}

\begin{proof} It is obvious that the first two implications of ($a$) are true in $\mathbf{ZF}$. Thus, it follows directly from Theorem \ref{s5t10} (i)--(ii) that ($a$) holds. To prove ($b$), let us notice that, in the light of Theorem \ref{s3t1}, we can fix a symmetric model $\mathcal{M}$ of $\mathbf{ZF}+\mathbf{CH}+\mathbf{WO}(\mathcal{P}(\mathbb{R}))+\neg\mathbf{CAC}_{fin}$. It follows from ($a$) that $\mathbf{Part}(\mathbb{R})$ is true in $\mathcal{M}$ but $\mathbf{M}(C, \hookrightarrow [0, 1]^{\mathbb{R}})$ fails in $\mathcal{M}$. 
\end{proof}

\begin{remark}
\label{r03.21}
 (i) To show that $\mathbf{Part}(\mathbb{R})$ is not provable in $\mathbf{ZF}$, let us recall that, in \cite{hkt}, a $\mathbf{ZF}$-model $\Gamma $ was constructed such that, in $\Gamma$, there exists a family $\mathcal{F}=\{F_{n}:n\in \mathbb{N}\}$
of two-element sets such that $\bigcup \mathcal{F}$ is a partition of $%
\mathbb{R}$ but $\mathcal{F}$ does not have a choice function. Then, in $\Gamma$, there does not exist an injection $\psi:\bigcup\mathcal{F}\to\mathbb{R}$ (otherwise, $\mathcal{F}$ would have a choice function in $\Gamma$). Hence $\textbf{Part}(\mathbb{R})$ fails in $\Gamma$.\smallskip

(ii) Since $\mathbf{Part}(\mathbb{R})$ is independent of $\mathbf{ZF}$, it follows from Theorem %
\ref{s5t9}(ii) that it is not provable in $\mathbf{ZF}$ that every compact metrizable subspace of the cube $[0,1]^{\mathbb{%
R}}$ has a base of size $\leq |\mathbb{R}|$.  We do not know if $\mathbf{M}(C, \hookrightarrow \lbrack 0,1]^{\mathbb{R}})$
implies every compact metrizable subspace of the cube $[0,1]^{\mathbb{R}}$
has a base of size $\leq |\mathbb{R}|$.\smallskip 

(iii) It is not provable in $\mathbf{ZFA}$ that every compact metrizable space with a
unique accumulation point embeds in $[0,1]^{\mathbb{R}}$. Indeed, in the Second Fraenkel model $\mathcal{N}$%
2 of \cite{hr}, there exists a disjoint family of two-element sets $%
\mathcal{A}=\{A_{n}:n\in \mathbb{N}\}$ whose union has no denumerable
subset. Let $A=\bigcup\mathcal{A}$, $\infty\notin A$, $X=A\cup\{\infty\}$ and, for every $n\in\mathbb{N}$, let $\rho_n$ be the discrete metric on $A_n$. Let $d$ be the metric on $X$ defined by ($\ast$) in Subsection \ref{s2.3}. Let $\mathbf{X}=\langle X, \tau(d)\rangle$. Then $\mathbf{X}$ is a compact metrizable space having $\infty$ as its unique accumulation point. Since, in $\mathcal{N}2$, the set
 $[0,1]^{\mathbb{R}}$ is well-orderable, while $\mathcal{A}$ has no choice function, it follows that $\mathbf{X}$ does not embed in the Tychonoff cube $[0,1]^{%
\mathbb{R}}$. This shows that the statement ``There exists a compact metrizable space with a unique accumulation point which is not embeddable in $[0,1]^{\mathbb{R}}$'' has a permutation model.
\end{remark}

\section{The list of open problems}
\label{s6}
For readers' convenience, let us repeat the open problems mentioned in Sections \ref{s4} and \ref{s5}.

\begin{enumerate}
\item Is $\mathbf{M}(C, WO)$ equivalent to or weaker than $\mathbf{M}(TB, WO)$ in $\mathbf{ZF}$?
\item  Does $\mathbf{M}(C, S)$ imply $\mathbf{CUC}$ in $\mathbf{ZF}$?
\item  Does $\mathbf{BPI}$ imply $\mathbf{CUC}$ in $\mathbf{ZF}$?
\item  Does $\mathbf{CUC}$ imply $\mathbf{M}(C, S)$ in $\mathbf{ZF}$?
\item  Does $\mathbf{M}(C, \hookrightarrow \lbrack
0,1]^{\mathbb{R}})$ imply $\mathbf{CAC}_{fin}$ in $\mathbf{ZF}$? (Cf. Question \ref{s5q7}(iv).)
\end{enumerate}

\end{document}